\documentclass[a4paper,12pt]{amsart}
\usepackage{amsmath,amsfonts,amssymb,amsthm}
\usepackage{latexsym,graphicx}
\usepackage{xypic}
\usepackage[matrix,arrow,ps,color,line,curve,frame]{xy}
\usepackage{tikz}

\newtheorem{proposition}{\sc Proposition}[section]
\newtheorem{lemma}[proposition]{\sc Lemma}

\newtheorem{theorem}[proposition]{\sc Theorem}
\newtheorem{definition}[proposition]{\sc Definition}
\theoremstyle{definition}

\theoremstyle{remark}
\newtheorem{remark}[proposition]{\sc Remark}

\newcommand{\can}{\operatorname{can}}
\newcommand{\alg}{\operatorname{Alg}}
\newcommand{\id}{\operatorname{id}}
\newcommand{\Ker}{\operatorname{Ker}}
\renewcommand{\ker}{\operatorname{Ker}}

\def\sw#1{{\sb{(#1)}}}
\def\eps{{\epsilon}}

\newcommand{\red}{\mbox{${}_B\!\mathop{\mbox{\rm Red}}^{H/J}(P)$}}
\newcommand{\ot}{\otimes}
\renewcommand{\phi}{\varphi}
\renewcommand{\epsilon}{\varepsilon}

\def\lco{\!\!\!\phantom{I}^{co H/J}\!H}
\def\o{\sp{[1]}}
\def\t{\sp{[2]}}

\renewcommand{\[}{\begin{equation}}
\renewcommand{\]}{\end{equation}}
\newcommand{\proj}[2]{\mathbb{#1}P_\mathcal{T}^#2}

\newcommand{\comul}{\Delta}
\newcommand{\counit}{\varepsilon}
\newcommand{\co}{\,\mathrm{co}\,}
\newcommand{\ev}{\mathrm{ev}}
\def\C{{\Bbb C}}

\def\R{{\Bbb R}}

\def\im{{\rm Im}}
\def\id{{\rm id}}
\newcommand{\Sta}{{}^{\langle 1\rangle}}
\newcommand{\Stb}{{}^{\langle 2\rangle}}
\newcommand{\ls}[1]{\ell(#1)^{\langle 1\rangle}}
\newcommand{\rs}[1]{\ell(#1)^{\langle 2\rangle}}
\newcommand{\mZ}{{\mathbb{Z}}}

\newcommand{\tplz}{{\mathcal{T}}}

\newcommand{\ztwo}{{\mathbb Z/2\mathbb Z}}

\newcommand{\qcube}{{C(S^2_{\B R\mathcal{T}})}}

\newcommand{\B}[1]{{\mathbb #1}}
\newcommand{\qcp}[1]{{C(\B R P^{#1}_\mathcal{T})}}
\begin{document}
\numberwithin{proposition}{section}
\numberwithin{equation}{section}
\baselineskip=16pt
\author{Piotr~M.~Hajac}
\address{Instytut Matematyczny, Polska Akademia Nauk, ul.~\'Sniadeckich 8, Warszawa, 00--656 Poland} 
\email{pmh@impan.pl}
\author{Jan Rudnik}
\address{Instytut Matematyczny, Polska Akademia Nauk, ul.~\'Sniadeckich 8, Warszawa, 00--656 Poland}
\email{yarood@gmail.com}
\author{Bartosz~Zieli\'nski}
\address{Department of Computer Science, Faculty of Physics and Appplied Informatics, University of \L{}\'od\'z, Pomorska 149/153,
 \L{}\'od\'z, 90--236 Poland}
\email{bartosz.zielinski@uni.lodz.pl}
\title[Reductions of piecewise-trivial  
comodule algebras]{{\Large Reductions of piecewise-trivial}\\
\vspace*{2mm} 
{\Large principal comodule algebras}}
\maketitle
\begin{abstract}
Let $G'$ be a closed subgroup of a topological group $G$. A principal 
$G$-bundle $X$ is reducible to a locally trivial principal 
$G'$-bundle $X'$
if and only if there exists a local trivialisation of $X$ such that all 
transition functions take values in~$G'$.
We prove a noncommutative-geometric counterpart of this theorem. 
To this end, we employ the concept of a piecewise-trivial principal 
comodule algebra as a  
replacement of a locally trivial compact principal bundle. 
To illustrate  our theorem, first we define a new noncommutative
deformation of the $\mathbb{Z}/2\mathbb{Z}$-principal bundle
$S^2\rightarrow \mathbb{R}P^2$ that yields a piecewise-trivial
principal comodule algebra. It is the C*-algebra of a quantum cube
whose each face is given by the Toeplitz algebra. The 
$\mathbb{Z}/2\mathbb{Z}$-invariant subalgebra defines the C*-algebra
of a quantum $\mathbb{R}P^2$. It is given as a triple-pullback of
Toeplitz algebras. Next, we prolongate this noncommutative 
$\mathbb{Z}/2\mathbb{Z}$-principal bundle to a noncommutative
$U(1)$-principal bundle, so that the former becomes a reduction of
the latter thus instantiating our theorem. Moreover, using K-theory results, 
we prove that
the prolongated noncommutative bundle is not trivial.
%To enclose natural and geometrically interesting noncommutative 
%examples, 
%we use smash products (cocycle-free crossed products) 
%rather than tensor products as a generalisation of trivial principal 
%bundles. 
%These examples serve as a testing ground for our reduction theorem.
\end{abstract}

\tableofcontents

\section{Introduction and preliminaries}
\noindent
The concept of a reduction of a principal bundle is a fundamental tool of gauge theory. It is  crucial 
in topology as a measurement of reducibility/non-triviality of principal bundles, and 
pivotal in differential geometry
because many important  
structures on manifolds
can be formulated as reductions of their frame bundles. 
For instance, 
an orientation, a volume form and a metric  on a manifold $M$
correspond to reductions of the frame bundle $FM$ to a 
$GL_+(n,\R)$-, $SL(n,\R)$- and $O(n,\R)$-bundle, 
respectively (see \cite{kn63} for more details).  

It also plays an important role in physics
where physical fields as well as twistors are considered as sections of fiber bundles. In particular,
a Higgs field is described as a section of a fiber bundle that corresponds to a principal-bundle reduction
via Theorem~\ref{hus} (see~\cite{nr84}). Furthermore, the idea of holonomy
links reductions of principal bundles with physics: principal bundles are always reducible to holonomy
subbundles~\cite{kn63} and  
the Aharonov--Bohm experiment indicates that it is holonomy rather than  a Yang--Mills potential (i.e.\ 
connection) or a field strength (i.e.\ curvature) that corresponds directly to a given physical situation 
and can be measured~\cite{ct86,a-j83,a-j86}.

In noncommutative geometry, a quantum-group gauge theory is based on the concept of a compact
quantum principal bundle whose algebraic backbone is a principal comodule algebra (see \cite{bdh17}
and references therein). In this Hopf-algebraic framework, reductions of quantum principal bundles
are handled through the Hopf--Galois
Reduction Theorem \cite{s-p99,g-r99,qsng} which remarkably recovers classical 
Theorem~\ref{hus} by establishing the equivalence
of reduction ideals $I$ and appropriate equivariant algebra homomorphism.
The latter have a geometric meaning of global sections of the fibre
bundle associated to a principal $G$-bundle via the canonical action
$G\times G/G'\rightarrow G/G'$, where $G'$ is a reducing subgroup of $G$.
They turn out to be far more manageable than reduction ideals, and allow us to unravel
reductions of piecewise-trivial principal comodule algebras~\cite{HKMZ}.

The aim of this article is to provide a critirion for a reducibility
of piecewise-trivial principal comodule algebras. More precisely, given
a Hopf algebra $H$ with bijective antipode, an appropriate Hopf ideal 
$J$, and 
%a finite family of smash
% products $\{B_i\rtimes H\}_i$ assembled into a multipullback 
a principal $H$-comodule
algebra $P$, we claim that:

\noindent {\sc Theorem} 
{\em 
There exists an
ideal $I\subseteq P$ such that $P/I$ is a piecewise-trivial
principal $H/J$-comodule algebra if and only if there exists a 
piecewise trivialisation of $P$ (with respect to the same covering) such that all the associated transition
functions anihilate $J$, and its associated actions of $J$ on the algebras 
covering the subalgebra of coaction invariants are 
trivial.} 

Our paper is in the Hopf-algebraic framework of unital algebras. In this framework, or in its coalgebraic
extension~\cite{bh99,bh04,z-b05}, noncommutative fiber bundles received recently much 
attention~\cite{bs19a,bs19b,bs19c,ms20}. This algebraic setting forces us to consider
noncommutative analogues of piecewise triviality (finite closed coverings) 
of compact principal bundles rather than their local triviality (arbitrary open coverings).
As explained in detail in \cite{bhms07}, the former concept is more general than the latter when
going beyond Lie groups. Only recently, ideas from the celebrated Elliott's classification program
of C*-algebras (the Rokhlin dimension) 
allowed the breakthrough of discovering a working definition of the local triviality of compact 
quantum principal bundles~\cite{ghtw18}. Formulating and proving our main reduction theorem
in the C*-algebraic framework of locally trivial compact quantum principal bundles, linked to
Hopf--Galois theory via the Peter--Weyl functor~\cite{bdh17}, is a key research problem stemming from
this paper. Finally, let us also mention that, in the spirit of localizations as in algebraic geometry, 
locally cleft and locally trivial Hopf--Galois extensions were studied in~\cite{r-d98,afl19}.

We work over a fixed ground field~$k$. The unadorned tensor product stands
for the tensor product over this filed. The comultiplication, the counit
and the antipode of a Hopf algebra $H$  are denoted by $\Delta$, 
$\varepsilon$ and $S$, respectively. Our standing assumption is that $S$
is invertible. 
A right $H$-comodule algebra $P$ is a unital associative algebra  equipped with an $H$-coaction 
$ \Delta_P : P \rightarrow P \otimes H$ that is an algebra map.  
For a comodule algebra $P$, we call 
\begin{equation}
P^{\co H}:=\left\{p \in P\,|\,\Delta_P(p)=p \otimes 1\right\}
\end{equation}
the subalgebra of coaction-invariant elements in $P$. A left coaction
on $V$ is denoted by ${}_V\Delta$.
For comultiplications and coactions, 
we often employ the Heynemann-Sweedler
notation with the summation symbol suppressed: 
\begin{equation}
\Delta(h)=:h_{(1)}\otimes h_{(2)},\quad
\Delta_P(p)=:p_{(0)}\otimes p_{(1)},\quad
{}_V\Delta(v)=:v_{(-1)}\otimes v_{(0)}.
\end{equation}
The convolution product of $f$ and $g$ is denoted by
\begin{equation}
(f*g)(h):=f(h_{(1)})g(h_{(2)}).
\end{equation}
 Finally, 
we use the convention that ${}_A^C\mathrm{Hom}_B^D$ signifies 
$k$-linear homomorphisms that are left $A$-linear, right $B$-linear,
left $C$-colinear and right $D$-colinear.
If $M$ is a right comodule over a coalgebra $C$ and $N$ is a left $C$-comodule,
then we define their {\em cotensor product} as
\[
M\overset{C}{\Box}N:=\{t\in M\otimes N\;|\;(\Delta_M\otimes\id)(t)=
(\id\otimes{}_N\Delta)(t)\}.
\]
In particular, for a right $H$-comodule algebra $P$ and a left 
$H$-comodule $V$, we observe that $P\Box^H V$ is a left $P^{\co H}$-module
in a natural way.

Following this introduction, we proceed to preliminaries concerning reductions and prolongations
both for classical principal bundles and for principal comodule algebras. Then we move to the main section
containing the multi-lemma 
proof of our main theorem. We end the paper with two sections devoted to geometrically
motivated
examples of reductions of piecewise-trivial principal comodule algebras. The first example is built
on new noncommutative deformations of the real projective plane and the two-dimensional sphere.
More precisely, replacing the C*-algebra of functions on the unit disc by the Toeplitz algebra
viewed as the C*-algebra of functions on a quantum disc~\cite{kl93}, we define a triple-pullback 
noncommutative deformation of the principal $U(1)$-bundle
\[
S^2\underset{\ztwo}{\times}U(1)\longrightarrow \mathbb{R}P^2,
\]
and prove that it is a non-trivial but piecewise trivial comodule algebra reducible to the triple-pullback 
noncommutative deformation of the principal $\ztwo$-bundle $S^2\to\mathbb{R}P^2$. The second example
shows that, unlike in the classical setting, a trivial principal comodule algebra need not be always reducible.

\subsection{Reductions, prolongations and the local triviality
 of classical principal bundles} 

Let $X\rightarrow M$ be a principal $G$-bundle over $M$, and $G'$ be 
a closed subgroup of $G$. A $G'$-{\em reduction} of $X\rightarrow M$
is a subbundle $X'\subseteq X$ 
over $M$ that is a principal $G'$-bundle over $M$ via the restriction
of the $G$-action on $X$. Reductions of a principal bundle to a given subgroup need not exist,
and when they exist, they are, typically, highly non-unique. In particular, a reduction to the trivial
subgroup is tantamount to a global section. Only trvializable principal bundles admit  global sections,
and there are as many global sections as trivializations.
\begin{proposition}\label{trivial}
A principal $G$-bundle $X$ is isomorphic as a $G$-space with
$X/G\times G$ if and only if there exists a continuous $G$-equivariant map $\Phi:X\rightarrow G$.
Then the isomorphism is given explicitly by
\begin{gather}
X\ni x\longmapsto([x],\Phi(x))\in X/G\times G,\quad X/G\times G\ni ([x],g)\longmapsto x\Phi(x)^{-1}g\in X.
\end{gather}
\end{proposition}

Much like differentiation is the inverse operation to integration, which is much easier and more algorithmic
than integration, the prolongation of a principal bundle is the always possible and uniqely determined
inverse operation to a reduction. More precisely, if $X'\to M$ is a principal $G'$-bundle, and 
$G'$ is a closed subgroup of $G$, then the \emph{$G$-prolongation} of $X'$ is the principal $G$-bundle
\begin{equation}
X'\underset{G'}{\times}G:=(X'\times G)/G',\quad \text{where\ }\forall\;x\in X', g\in G, h\in G'\colon
 (x,g)h:=(xh,h^{-1}g).
\end{equation}

For instance, the Ehresmann groupoid $G\times_{G'}G$  can be thought of as the $G$-prolongation of 
$G$ treated as the
principal $G'$-bundle. It is trivializable as a $G$-bundle due to  Proposition~\ref{trivial} because
we can define  
\begin{equation}
\Phi\colon G\underset{G'}{\times}G\ni [(g,h)]\longmapsto gh\in G.
\end{equation} 

Now we can state:
\begin{proposition}
Let $G'$ be a closed subgroup of $G$. Assume that  a principal $G$-bundle $X\to M$ is reducible 
to a principal $G'$-bundle $X'\to M$.
Then 
\begin{gather}
X\ni x\longmapsto [(x',g)]\in X'\underset{G'}{\times}G,\quad \text{where\ } x'g=x,\label{isoeresman}\\
X'\underset{G'}{\times}G\ni [(x',g)]\longmapsto x'g,
\end{gather}
is a pair of mutually inverse gauge isomorphisms ($G$-equivariant homeomorphisms inducing identity
on $M$).
\end{proposition}

Next, recall that reductions of principal bundles are classified 
by  global sections 
of appropriate associated fibre bundles \cite[Theorem~2.3]{h-d94} (see Theorem~\ref{hus}). 
More precisely, a $G$-principal bundle $X\rightarrow M$ 
can be reduced to a $G'$-subbundle if and only if there exists 
a global section of the associated fibre bundle $X/G'\rightarrow M$. There is a natural way 
to provide a one-to-one correspondence between the $G'$-reductions of $X$ and the global sections of 
$X/G'$. This correspondence 
 supports the geometric intuition of a  $G'$-subbundle as a $G'$-thick global section 
of~$X$. 
%Furthermore, the group 
%inverse allows us to identify $G/G'$ with $G'\backslash G$ and $G$-equivariant maps into $G/G'$ 
%with $G$-equivariant maps into $G'\backslash G$: $f:X\rightarrow G'\backslash G$, $f(xg)=f(x)g$. 
\begin{theorem}\label{hus}
Let $G'$ be a closed subgroup of $G$. 
A principal $G$-bundle $X$ is reducible to a principal $G'$-bundle $X'$ if and only if
there exists a continuous $G$-equivariant map \mbox{$f:X\rightarrow G'\backslash G$}.
Explicitly, given the map $f$, the reduced subbundle can be recovered as 
\begin{equation*}
X':=f^{-1}([e]). 
\end{equation*}
Vice versa, having a $G'$-reduction $X'$, we can construct an appropriate map 
$f$ by composing the isomorphism
\eqref{isoeresman} with the projection on the second component and the quotient map:
\begin{equation*}
X\ni x\longmapsto [(x',g)]\longmapsto [g]\in G'\backslash G.
\end{equation*}
\end{theorem}
\noindent
The Hopf--Galois Reduction Theorem (Theorem~\ref{hogare}) is a noncommutative counterpart of the above
result.

All the foregoing discussion of reductions and prologations
is valid for arbitrary Cartan principal bundles~\cite{c-h67}, 
i.e.\ without the assumption of
local triviality. The local triviality of a principal bundle allows us to phrase its reducibility
 in terms of 
transition functions (cf.\ \cite{kn63}, Proposition~I.5.3):
\begin{theorem}
Let $G'$ be a closed subgroup of  $G$. A principal $G$-bundle $X\to M$ is reducible to 
a locally trivial principal $G'$-bundle $X'\to M$
if and only if there exists a local trivialisation of $X$ (with respect to the same covering as that of $X'$)
such that all transition functions take values in~$G'$.
\end{theorem}
\noindent
A noncommutative-algebraic counterpart of this theorem, i.e.\ Theorem~\ref{mainres},
 is the main result of this paper. 
%We refer to it
%as the Hopf--Galois Reduction Theorem (Theorem~\ref{mainres}). 
We derive it by restricting the above
theorem to compact Hausdorff spaces and replacing local triviality with the more general concept of
piecewise triviality~\cite{bhms07} (finite closed coverings instead of arbitrary open coverings), 
then using the Peter--Weyl functor \cite{bhms07,bh14} to
translate everything into algebraic terms of Hopf--Galois 
theory, and finally proving it
for arbitrary principal comodule algebras.

In particular, the structure groups of trivial principal bundles can be reduced 
to arbitrary subgroups. However, a reduction of a trivial principal bundle need not be trivial.
For instance, the boundary of the M\"obius strip viewed as a non-trivial 
$\mathbb{Z}/2\mathbb{Z}$-bundle over the unit circle $S^1$  
can be obtained as a reduction of the trivial $ U(1) $-bundle over $S^1$. 
According to Theorem~\ref{hus}, the reductions of $S^1\times U(1)$ are in one to one correspondence with
continuous $U(1)$-equivariant maps $f:S^1\times U(1)\rightarrow (\mZ/2\mZ)\backslash U(1)$.
Let us consider two choices of such maps:
\begin{gather}
f_1:S^1\times U(1)\ni (s,u)\longmapsto [su]\in (\mZ/2\mZ)\backslash U(1),\\
f_2:S^1\times U(1)\ni (s,u)\longmapsto [s^{1/2}u]\in (\mZ/2\mZ)\backslash U(1).
\end{gather}
It is easy to verify that $f_1^{-1}([e])\cong S^1\times \mZ/2\mZ$. Explicitly,
$f_1^{-1}([e])=\{(\pm u,u^{-1})\;|\;u\in U(1)\}$, where we identify $S^1$ with $U(1)$.
Note that the action of $\mZ/2\mZ$ on $f_1^{-1}([e])$ sends an element of one circle $(u,u^{-1})$
to the element $(u,-u^{-1})=(-(-u),(-u)^{-1})$ which belongs to the other circle.
In the other case, $(s,u)\in f_2^{-1}([e])$ if and only if $s^{-1/2}u=\pm e$, i.e.\ $s=u^2$, whence
$f_2^{-1}([e])$ is homeomorphic with $S^1$, with an explicit formula given by $u\mapsto (u^2,u)$.
This time, the action of $\mZ/2\mZ$ sends $u$ to $-u$. 
It is easy to see that $S^1$ with this action
is the edge of the M\"obius strip viewed as a principal $\mZ/2\mZ$-bundle.

Therefore, one has to bear in mind that a local trivialisation
of a principal $G$-bundle $X$ when restricted to a reduced 
$G'$-subbundle $X'$ need not be a trivialisation of~$X'$.
The point here is that 
the principal bundle $U(1) \rightarrow U(1)/(\mathbb{Z}/2\mathbb{Z})$ 
is not trivial. 
Its triviality 
would be a sufficient condition for the triviality of the reduction:
\begin{proposition}\label{simple}
If $ G \rightarrow G/G'$ is trivial as a principal $G'$-bundle, then any 
$G'$-reduction of a~trivial 
$G$-bundle is trivial. 
\end{proposition}

\subsection{Reductions and prolongations of principal comodule algebras} 
\label{principalprel}
Let $H$ be a Hopf algebra, $P$ be a right $H$-comodule algebra, and 
let $B:=P^{\co H}$ be the coaction-invariant
subalgebra. The $H$-comodule algebra $P$ is called a {\em principal} 
\cite{bh04} iff: 
\begin{enumerate}
\item
the canonical map
$P{\otimes}_B P\ni p \otimes q \mapsto\can(p \otimes q):=
pq_{(0)} \otimes q_{(1)}\in P \otimes H$
is bijective,
\item
$\exists\; s\in {}_B\mathrm{Hom}^H(P,B\otimes P):\; m\circ s=\id$,
where $m$ is the multiplication map,
\item
the antipode of $H$ is bijective.
\end{enumerate}
Here (1) is the Hopf--Galois (freeness) condition, (2) means the equivariant
projectivity of $P$ (equivalent to faithful flatness), and (3) ensures a left-right symmetry of the
 definition (everything can be re-written for left comodule
algebras).
The inverse of the canonical map can be written explicitly using the
Heynemann-Sweedler like notation: 
$\can^{-1}(p\otimes h):=ph\o\otimes_B h\t$. The restriction of this map
\begin{equation}
H\ni h\longmapsto\can^{-1}(1\otimes h)
=:h\o\underset{B}{\otimes} h\t
\in P\underset{B}{\otimes}P
\end{equation}
is called the {\em translation map}. It enjoys the following property which
we will use later on:
\begin{equation}\label{TransProp}
h\o h\t=\counit(h).
\end{equation}
%For more details see, e.g., \cite{m-s93,s-hj94}. 

If $H$ is a Hopf algebra with bijective antipode and $P$ is 
a right $H$-comodule algebra, then one can show (cf.~\cite{bh04})
that it is
principal if and only if there exists a linear map
\begin{equation}
\ell:H\longrightarrow P\otimes P, \quad h\longmapsto
\ell(h)=:\ls{h}\otimes \rs{h},
\end{equation}
that, for all $h\in H$, satisfies:
\begin{gather}
\ls{h}\rs{h}\sw{0}\otimes\rs{h}\sw{1}=1\otimes h,\\
S(h\sw{1})\otimes\ls{h\sw{2}}\otimes \rs{h\sw{2}}=
\ls{h}\sw{1}\otimes\ls{h}\sw{0}\otimes\rs{h},\\
\ls{h\sw{1}}\otimes\rs{h\sw{1}}\otimes h\sw{2}
=\ls{h}\otimes\rs{h}\sw{0}\otimes\rs{h}\sw{1}.
\end{gather}
%$\ls{h}\rs{h}=\counit(h)$
Any such a map $\ell$ can be made unital \cite{bh04}. It is 
then called a {\em strong connection} \cite{h-pm96,dgh01,bh04}, 
and can be thought of as a unital bicolinear lifting of 
the translation map. To any strong connection we can associate the 
left $B$-linear right $H$-colinear splitting of the multiplication map:
\begin{equation}\label{splitting}
s:P\ni p\longmapsto p\sw0\ell(p\sw1)\Sta\otimes \ell(p\sw1)\Stb\in B\otimes P.
\end{equation}

A particular class of principal comodule algebras is distinguished by the 
existence of a cleaving map. A cleaving map is defined as a unital 
right $H$-colinear
convolution-invertible map \mbox{$j:H\rightarrow P$}. Having a cleaving map,
one can define a strong connection as
\[\label{cleftstrong}
\ell:=(j^{-1}\otimes j)\circ\Delta
,
\]
where $j^{-1}$ stands for the convolution inverse of~$j$.
Comodule algebras admitting
a cleaving map are called {\em cleft}. All modules associated with cleft
comodule algebras are always free. Also, one can show that a cleaving map is
automatically injective. Therefore, as the value of a cleaving map on
a group-like element is invertible, we can conclude that the existence of
a non-trivial group-like in $H$ necessitates the existence of an invertible
element in $P$ that is not a multiple of~$1$. Hence, one of the ways to prove
the non-cleftness of a principal comodule algebra over a Hopf algebra with a 
non-trivial group-like is to show the lack of non-trivial invertibles in
the comodule algebra.

If $j:H\rightarrow P$ is a right $H$-colinear algebra homomorphism, then
it is automatically convolution-invertible and unital. A cleft comodule
algebra admitting a cleaving map that is an algebra homomorphism is called
a {\em smash product}. All commutative smash products reduce to the
tensor algebra $P^{\co H}\otimes H$, so 
smash products play the role of trivial bundles. A cleaving map defines
a left action of $H$ on $P^{\co H}$ making it a left $H$-module algebra:
$h\triangleright p:=j(h_{(1)})pj^{-1}(h_{(2)})$. Conversely, if $B$ is
a left $H$-module algebra, one can construct a smash product 
$B\rtimes H$ by equipping the vector space $B\otimes H$ with 
the multiplication
\[
(a\otimes h)(b\otimes k):=a\,(h_{(1)}\triangleright b)\otimes h_{(2)}\,k, \ \ 
a,b\in B,\ h,k\in H,
\]
and coaction $\Delta_{B\rtimes H}:=\id\otimes\Delta$. Then a cleaving map
is simply given by $j(h)=1\otimes h$. Plugging it into the formula 
\eqref{cleftstrong} yields
 a  strong connection defined by
\begin{equation}
\label{smashstrong}
\ell:H\longrightarrow(B\rtimes H)\otimes(B\rtimes H),\quad 
h\longmapsto(1\otimes S(h\sw{1}))\otimes(1\otimes h\sw{2}).
\end{equation}

To end with, let us recall crucial facts about reductions and prolongations of principal comodule algebras.
\begin{definition}[\cite{g-r99,s-p99,qsng}]\label{red}
Let $P$ be a principal $H$-comodule  algebra with $B:=P^{\co H}$ and $J$
 be a Hopf ideal of $H$ such that $H$ is a principal left $H/J$-comodule
algebra. 
We say that an ideal $I$ of $P$ is a \emph{$J$-reduction}
 of $P$ iff the following conditions are satisfied:
\begin{enumerate}
\item $I$ is an $H/J$-subcomodule of $P$,
\item $P/I$ with the induced coaction is a principal  $H/J$-comodule algebra,
\item $(P/I)^{\co H/J}=B$.
\end{enumerate}
\end{definition}
\noindent
Losely speaking, $J$ plays the role of the ideal of functions vanishing on a subgroup and 
$I$ plays the role of the ideal of functions vanishing on a subbundle. Thus, $H/J$ works as the algebra of the 
reducing subgroup and $P/I$ works as the algebra of the reduced bundle. 
The coaction invariant subalgebra $B$ 
remains intact --- the base space of a subbundle coincides with the base space of the bundle. 

Prolongations of principal comodule algebras are given as cotensor products naturally describing
the associated fiber bundle construction~\cite{bhms07}. 
As in the classical case, if $P/I$ is a principal $H/J$-comodule algebra such that $I$ is a $J$-reduction, then
\begin{equation}
P/I\overset{H/J}{\Box}H\cong P\quad\text{as H-comodule algebras.}
\end{equation}

The space of all such $J$-reducing ideals we denote by 
${}_B\mbox{Red}^{H/J}(P)$. It can happen that this set contains only 
the zero ideal, as for a given non-zero $J$ there need not exist a reduction. If no 
non-zero $J$ admits a reduction, 
we say that the principal comodule algebra is {\em irreducible}. 
The thus defined reductions have a clear conceptual meaning but are
 difficult
to handle. Following the classical case (see Theorem~\ref{hus}), one can 
prove that they are equivalent to right
$H$-colinear 
algebra homomorphisms from the left coaction-invariant subalgebra 
${}^{co H/J}H$
to the centralizer subalgebra
\begin{equation}
Z_P(B):=\{p\in P\;|\; pb=bp,\;\forall\; b\in B\}
\end{equation}
 that are compatible with the
Miyashita--Ulbrich action.
 The latter condition (trivial in the commutative case) 
means that 
\begin{equation}
f(S(h\sw1)kh\sw2)=h\o f(k)h\t,\;\forall\; k\in\lco,\; h\in H.
\end{equation}
The space of all such homomorphisms we denote by
$\mbox{Alg}^H_H(\!\!\!\phantom{I}^{co H/J}\!H,Z_P(B))$.
Note that $S(h\sw1)kh\sw2\in \lco$ for all $k\in\lco$, $h\in H$.
%  Note also that
% $\alg^H_H(\lco,Z_P(B))$ is  an 
% \ga-space via the action $f\mapsto F\circ f$.

\begin{theorem}[Hopf--Galois Reduction \cite{g-r99,s-p99,qsng}]
\label{hogare}
Let $P$ be a 
principal $H$-comodule algebra, and $B:=P^{\co H}$.
Then the formulas
\begin{eqnarray}
\alg^H_H(\lco,Z_P(B))\ni f&\longmapsto& I_f:=Pf(\lco\cap\Ker\varepsilon)\in\red,\\
\red\ni I&\longmapsto& f_I\in\alg^H_H(\lco,Z_P(B)),\nonumber\\
f_I(k)&:=&S^{-1}(k)\o (i_B\circ\pi_I)(S^{-1}(k)\t),\\
i_B(\pi_I(b+x))&:=&b,\;\;\; i_B: (B\oplus I)/I\rightarrow B,\;\;\; b\in B,\;\;\; x\in I,\nonumber
\end{eqnarray}
define mutually inverse bijections.
\end{theorem}
\noindent
One can treat this theorem as a source of the alternative definition of reductions of principal 
comodule algebras given as elements of $\alg^H_H(\lco,Z_P(B))$.
 
\section{Reductions of piecewise-trivial comodule algebras}\label{redupt}

\subsection{Piecewise triviality revisited}

%\begin{definition} (cf. \cite[Definition~3.7]{HKMZ})
A family of surjective algebra homomorphisms\linebreak
\mbox{$\{\pi_i:P\rightarrow P_i\}_{i\in\{1,\ldots,N\}}$} is 
called a {\em covering} \cite[Definition~3.6]{HKMZ} when
\begin{enumerate}
\item $\bigcap_{i\in\{1,\ldots,N\}}\ker\pi_i=\{0\}$,
\item the family of ideals $(\ker\pi_i)_{i\in\{1,\ldots,N\}}$ generates a distributive lattice with 
$+$ and $\cap$ as meet and join, respectively.
\end{enumerate}
%\end{definition}

Next,
let $\{\pi_i:P\rightarrow P_i\}_{i}$ be a covering. We define
the family of canonical surjections
\begin{equation}\label{pij}
\pi^i_j:P_i\rightarrow P/(\ker\pi_i+\ker\pi_j),
\quad \pi_i(p)\mapsto p+\ker\pi_i+\ker\pi_j,
\end{equation}
and denote by $P^c$ the multipullback of $P_i$'s along $\pi^i_j$'s: 
\begin{equation}\label{pcdef}
P^c:=\{(p_i)_i\in\bigoplus_iP_i\;|\;\pi^i_j(p_i)=\pi_i^j(p_j)\}.
\end{equation}
The following Proposition states the relationship between $P$ and $P^c$. 
\begin{proposition}[\cite{CM00}] \label{pullisom2}%(cf. \cite[Proposition~2]{CM00}) 
Let $\{\pi_i:P\rightarrow P_i\}_{i\in\{1,\ldots,N\}}$ be a covering. 
Then the map 
\begin{equation}\label{pullisom}
\chi:P\longrightarrow P^c,\quad p\longmapsto(\pi_i(p))_i\,,
\end{equation}
is an algebra isomorphism. (If $P$ and all the $P_i$'s are $H$-comodule
 algebras for some 
Hopf algebra $H$, and all the $\pi_i's$ are  colinear,
 then so is $\chi$.)
\end{proposition}
\noindent
The isomorphism~\eqref{pullisom} is what makes the notion of the covering
 so much useful, as it often allows us to glue the properties of the parts of $P$ (the $P_i$'s) into the properties 
of the whole~$P$.

We recall now the notion of a quantum version of piecewise triviality of the bundle (which is like local triviality, 
but with respect to closed subsets):

\begin{definition}[{\cite[Definition~3.8]{HKMZ}}] \label{PiecePrinc}
An $H$-comodule algebra $P$ is called piecewise trivial
 if there exists a 
family of surjective $\{\pi_i:P\rightarrow P_i\}_{i\in\{1,\ldots,N\}}$ 
$H$-colinear maps such that:
\begin{enumerate}
\item the restrictions 
$\pi_i|_{P^{\co H}}:P^{\co H}\rightarrow P_i^{\co H}$ form a covering,
\item the $P_i$'s are  smash products ($P_i\cong P_i^{\co H}\rtimes H$ 
as $H$-comodule algebras).
\end{enumerate}
\end{definition}
\noindent
Note that, if the antipode of $H$ is bijective, then it follows from
the main result of \cite{HKMZ} that $P$ is principal --- this is an important instance of gluing of properties 
mentioned above. 
To emphasize
this fact and stay in touch with the classical terminology, we frequently
use the phrase ``piecewise-trivial principal comodule algebra".

Note also that the consequence of principality of $P$ is that 
$\{\pi_i:P\rightarrow P_i\}_{i\in\{1,\ldots,N\}}$ is a covering of $P$. 
To see this, one can use \cite[Proposition~3.4]{HKMZ} which states that $K\mapsto K\cap P^{\co H}$
is a lattice monomorphism between the lattice of ideals in $P$ which are right $H$-comodules and the lattice of 
ideals in
$P^{\co H}$. Indeed, we have  that 
\begin{equation}
P^{\co H}\cap\bigcap_i\ker\pi_i=\bigcap_i(\ker\pi_i\cap P^{\co H})
=0
\end{equation}
 by assumption,  so $\bigcap_i\ker\pi=0$ by the injectivity of $P^{\co H}\cap{\cdot}$.
The distributivity follows much in the same way as $P^{\co H}\cap{\cdot}$ maps monomorphically 
the lattice generated by $\ker\pi_i's$ into a distributive lattice.

The following lemma is the slight generalization of  the result implicit in the proof of
 \cite[Proposition~3.4]{HKMZ}.
It is used in the proof of our main result, but it is also interesting on its own.
\begin{lemma}\label{lplemma}
Let $P$ be a principal $H$-comodule algebra and $B=P^{\co H}$. Let $K$ be an ideal and
a right $H$-subcomodule of $P$, and let $L$ be an ideal in $B$.
Then $L=K\cap B$ if and only if $K=LP$.
%Let $P$ be a principal $H$-comodule algebra and $B=P^{\co H}$. Suppose that an ideal and a right $H$-subcomodule $K$
%of $P$ is of the form $K=LP$ where $L$ is an ideal in $B$. Then $L=K\cap B$.
\end{lemma}
\begin{proof}
Assume first that $K=LP$. 
 It is obvious that
$L\subseteq B\cap K$. To prove the converse inclusion, take
any $p:=\sum_il_ip_i\in K\cap B$, where $l_i\in L$, $p_i\in P$,
 for all $i$. Taking advantage of the splitting \eqref{splitting}
provided by a strong connection and 
any unital linear functional $f$ on $P$,  we compute
\begin{equation}
p=p\ell(1)\Sta f(\ell(1)\Stb)=p\sw0\ell(p\sw1)\Sta f(\ell(p\sw1)\Stb)=\sum_il_ip_i\sw0\ell(p_i\sw1)\Sta 
f(\ell(p_i\sw1)\Stb). 
\end{equation}
Hence, $p\in L$ as $p_i\sw0\ell(p_i\sw1)\Sta f(\ell(p_i\sw1)\Stb)\in B$ and $L$ is an ideal in $B$.

Conversely, assume that $L=B\cap K$. 
The inclusion $LP\subseteq K$ is obvious because
$K$ is an ideal in $P$. To show the opposite inclusion,
apply the splitting~\eqref{splitting} to any $p\in K$. Then
\begin{equation}
B\otimes P\ni p\sw0\ell(p\sw1)\Sta\otimes\ell(p\sw1)\Stb\in K\otimes P
\end{equation}
 because $K$ is a subcomodule and an ideal in $P$.
Therefore, 
\begin{equation}
p=p\sw0\eps(p\sw1)=p\sw0\ell(p\sw1)\Sta\ell(p\sw1)\Stb\in 
(B\cap K)P= LP,
\end{equation}
as needed.
\end{proof}

Finally, we recall  quantum versions of the the concepts
of a piecewise trivialisation and transition functions:
\begin{definition}
Let $\{\pi_i:P\to P_i\}_i$ be a covering by right $H$-colinear maps
of a  principal right $H$-comodule algebra
$P$ such that the restrictions 
$\pi_i|_{P^{\co H}}:P^{\co H}\rightarrow P_i^{\co H}$ also
form a covering.
A {\em piecewise trivialisation} of $P$ with respect to 
the covering $\{\pi_i:P\to P_i\}_i$  is a family
$\{\gamma_i:H\to P_i\}_i$ of right $H$-colinear algebra homomorphisms
(cleaving maps).
\end{definition}
\noindent
It is clear that a principal comodule algebra is piecewise-trivial
if and only if it admits a piecewise trivialisation.
With each piecewise trivialisation of $P$ we can associate the {\em
transition functions} 
\begin{equation}
\label{trfundef}
T_{ij}:=(\pi^i_j\circ\gamma_i)*(\pi^j_i\circ\gamma_j\circ S):\;
H\longrightarrow P/(\ker\pi_i+\ker\pi_j),
\end{equation}
where $\pi^i_j$'s are given by~\eqref{pij}.
It follows directly from the colinearity of $\pi^i_j$'s and $\gamma_j$'s
 that the elements in the images of all the $T_{ij}$'s are
  coaction invariant. Combining this with the fact that intersecting
kernels of $\pi_j$'s with coaction invariant subalgebra defines
a homomorphism of lattices \cite[Proposition~3.4]{HKMZ}, we conclude
 that the image of each $T_{ij}$
  is contained in 
$P^{\co H}/(\ker\pi_i|_{P^{\co H}}+\ker\pi_j|_{P^{\co H}})$.

As in the classical setting, transition functions can be used to 
assemble a principal comodule algebra from trivial pieces. Indeed,
\eqref{pcdef} can be rewritten as
\begin{equation}
\label{pctrans}
P^c=\{(p_i)_i\in\prod_iP_i\;|\;\pi^i_j(p_i\sw{0}\gamma_i(S(p_i\sw{1})))
T_{ij}(p_i\sw{2})
\otimes p_i\sw{3}
=\pi^j_i(p_j\sw{0}\gamma_j(S(p_j\sw{1})))\otimes p_j\sw{2}\}.
\end{equation}
Since, for any $i$ and $j$, we have
\begin{equation}
\im\;T_{ij}\subseteq P^{\co H}/(\ker\pi_i|_{P^{\co H}}+\ker\pi_j|_{P^{\co H}})
\end{equation}
 and 
$p\sw{0}\gamma_i(S(p\sw{1}))\in P_i^{\co H}$, the 
compatibility conditions defining $P^c$ all take place at the 
base-space (coaction invariant) algebras.

We are now ready to state the main result of this paper:
\begin{theorem}\label{mainres}
Let $P$ be a principal right $H$-comodule algebra, 
and $J$ a Hopf ideal of $H$ such that $H$ is a principal left 
$H/J$-comodule algebra. 
Then there exists a $J$-reduction of $P$ to a piecewise-trivial principal right $H/J$-comodule algebra if and only if
there exists a 
piecewise trivialisation of $P$ (with respect to the same covering 
$\{B\rightarrow B_i\}_{i\in\{1,\cdots,N\}}$
as that of the $J$-reduction)
 such that  $T_{ij}(J)=0$ 
for all the associated transition
functions $T_{ij}$ and $J\triangleright_i B_i=0$ for all the actions
$H\otimes B_i\rightarrow B_i$, 
$h\triangleright_i b:=\gamma_i(h_{(1)})b\gamma_i(S(h_{(2)}))$.
% for all $h\in J$, $b\in B_i$, for any index $i$.
\end{theorem}

\subsection{A proof of the main theorem}

Our proof consists of two parts each of which establishes one of the implications of the asserted equivalence.
Both parts are
divided into several lemmas. 
 First, we provide lemmas needed for proving the implication
``the existence of a trivialisation with some properties implies that there exists a reduction to a 
piecewise-trivial comodule algebra''.

Our first lemma is a certain general statement needed in the second 
lemma.
\begin{lemma}[\cite{qsng}]\label{jl}
Let $L$ be a bialgebra and $\overline{L}$ be a coalgebra and a left $L$-module. Assume that there 
exists a surjective left $L$-linear coalgebra map $\pi: L\rightarrow \overline{L}$, and view $L$ 
as a left $\overline{L}$-comodule with the coaction ${}_L\Delta = (\pi\ot\id)\circ \Delta$. 
Then 
\begin{equation}
\label{desc.coin} D:=\,^{co\overline{L}}\!L
 = \{d\in L\; |\;{}_L\Delta(d) = \pi(1)\otimes d\}
\end{equation} 
is a right $L$-comodule subalgebra of $L$, i.e.\ 
$\Delta(D)\subseteq D\otimes L$. Furthermore, the augmentation ideal
$D^+:=D\cap \ker\eps$ is contained in
$\ker\pi$ and $\Delta(d)-1\ot d\in D^+\ot L$ for all $d\in D$. 
\end{lemma}

In the following lemma, we prove the existence of a 
reduction of a trivial (smash product) comodule algebra
when the trivialising map satisfies certain condition.
\begin{lemma}
\label{trivredl}
Let $P$ be a smash product $H$-comodule algebra, $B:=P^{\co H}$,
and  $\gamma:H\rightarrow P$ be a cleaving map.
Let $J$ be a  Hopf ideal of $H$ such that 
$h\triangleright b:=\gamma(h_{(1)})b\gamma(S(h_{(2)}))=0$ 
for all $h\in J$ and $b\in B$. 
Then $\gamma$ restricts to an element of
$\alg^H_H({}^{\co H/J}H,Z_P(B))$. 
\end{lemma}
\begin{proof}
Denote for brevity $D:={}^{\co H/J}H$. By definition, $\gamma$ restricted to $D$ is in
$\alg^H(D,P)$. 
The translation map can be written in terms of $\gamma$ as
follows:
 $h^{[1]}\otimes_B h^{[2]}=\gamma(S(h\sw{1}))\otimes_B \gamma(h\sw{2})$.
Hence, the $H$-linearity of $\gamma$ for the Miyashita--Ulbrich
action follows directly from the fact that $\gamma$ is an algebra map.  It remains to show that
$\gamma(h)\in Z_P(B)$ for all $h\in D$. To this end, note that  
$D^+\subseteq J$ and $\Delta(D)\subseteq D\otimes H$ by Lemma~\ref{jl}.
 %Also by Lemma~\ref{jl},. \
%Hence it is enough to show that $\gamma(h)\in Z_P(B)$ for all $h\in J$.
Now, let $h\in D$ and $b\in B$. Then, using
 $\nu:D\ni h \mapsto h -\varepsilon(h)1_H\in D^+$, we obtain
\begin{equation}
\gamma(h)b=(h\sw{1}\triangleright b)\gamma(h\sw{2})
=b\gamma(h)+(\nu(h\sw{1})\triangleright b)\gamma(h\sw{2})
%\gamma(\nu(h\sw{1})\sw{1})b\gamma(S(\nu(h\sw{1})\sw{2}))\gamma(h\sw{3})
=b\gamma(h).
\end{equation}
This ends the proof.
\end{proof}

The next lemma  provides a way in which reductions can be combined 
together in a piecewise-trivial comodule algebra.
\begin{lemma}\label{redcomp}
Let $H$ be a Hopf algebra with bijective antipode and  
$J$ be a Hopf ideal of $H$
such that the antipode of $H/J$ is also bijective. Let
 $P$ be a piecewise-trivial principal 
$H$-comodule algebra with a covering
$\{\pi_i:P\rightarrow P_i\}_{i\in \{1,\ldots,N\}}$.
Denote $B_i:=P_i^{\co H}$ and $B:=P^{\co H}$. Then,
if there exists a family of maps 
$f_i\in \alg^H_H(\lco,Z_{P_i}(B_i))$, $i\in\{1,\ldots,N\}$, 
such that $\pi^i_j\circ f_i=\pi^j_i\circ f_j$ for all $i,j$,
the following map defined with the help of \eqref{pullisom}
\begin{equation}
f:\lco\longrightarrow P,\quad h\longmapsto\chi^{-1}((f_i(h))_i),
\end{equation}
is an element of $\alg^H_H(\lco,Z_{P}(B))$.
\end{lemma}
\begin{proof}
It is immediate that $f\in \alg^H(\lco,P)$. Furthermore,
 for any $h\in\lco$ and $b\in B$,
\begin{equation*}
bf(h)=b\chi^{-1}((f_i(h))_i)=\chi^{-1}((\pi_i(b)f_i(h))_i)=\chi^{-1}((f_i(h)\pi_i(b))_i)
=\chi^{-1}((f_i(h))_i)b=f(h)b,
\end{equation*}
so  $f(h)\in Z_P(B)$. 
Finally, if $\tau:H\rightarrow P\otimes_B P$ is the 
translation map for $P$, then $(\pi_i\otimes\pi_i)\circ \tau$ is the translation map for
 $P_i$ and, for any $k\in H$ and $h\in\lco$, we can compute:
\begin{align}
k^{[1]}f(h)k^{[2]}&=k^{[1]}\chi^{-1}((f_i(h))_i)k^{[2]}\\
&=\chi^{-1}((\pi_i(k^{[1]})f_i(h)\pi_i(k^{[2]}))_i)\nonumber\\
&=\chi^{-1}((f_i(Sk\sw{1}hk\sw{2}))_i)\nonumber\\
&=f(Sk\sw{1}hk\sw{2}).\nonumber
\end{align}
Hence, $f$ is an element of $\alg^H_H(\lco,Z_{P}(B))$.
\end{proof}

To combine the above two lemmas, we need the following.
\begin{lemma}
\label{trivcor}
Let $J$ be a Hopf ideal of $H$ and 
$\{\gamma_i:H\rightarrow P_i\}_{i\in\{1,\cdots,N\}}$ 
be a  piecewise trivialisation
of a principal $H$-comodule algebra~$P$. Then, 
$\forall\; i,j\in\{1,\cdots,N\}:$
\[
T_{ij}(J)=0
\;\Rightarrow\;
\forall\; h\in{}^{\co H/J}H: 
\pi^i_j(\gamma_i(h))=\pi^j_i(\gamma_j(h)),
\]
where $\pi^i_j$'s are the canonical surjections of~\eqref{pij}
and $T_{ij}$'s are the transition functions of~\eqref{trfundef}.
\end{lemma}
\begin{proof}
Denote for brevity $D:={}^{\co H/J}H$. 
For all $i$, $j$, the equality
 $\pi^i_j(\gamma_i(h))=\pi^j_i(\gamma_j(h))$ is equivalent
to $T_{ij}(h)=\varepsilon(h)$ because 
%taking the convolution product with 
$\gamma_j*(\gamma_j\circ S)=\epsilon=(\gamma_j\circ S)\circ\gamma_j$. 
Furthermore, $T_{ij}(J)=0$ by assumption and 
 $D^+\subseteq J$ by Lemma~\ref{jl}, so, for any
$h\in D$, we obtain
\begin{equation}
T_{ij}(h)=\varepsilon(h)+T_{ij}(h-\varepsilon(h))=\varepsilon(h).
\end{equation}
\end{proof}

The preceding three lemmas combined with
 Theorem~\ref{hogare} yield that $P/Pf(J)$ is an $H/J$-principal
comodule algebra. It remains to show that $P/Pf(J)$ is piecewise trivial. To this end, we apply 
Lemma~\ref{lplemma}
to show that a covering of $P$ induces a covering of $P/Pf(J)$.
For brevity, denote $Pf(J)$ by $I$. Let $[\cdot]:P\rightarrow P/I$
stand for the canonical surjection. 
Define $\bar{P}_i:=P_i/\pi_i(I)$ for all $i$.
The surjections $\pi_i$ descend to
$\bar{\pi}_i:P/I\rightarrow \bar{P}_i$. 
From Lemma~\ref{lplemma}, 
we conclude that
\begin{equation}\label{kers}
\ker\bar\pi_i=[\ker\pi_i]=[\ker\pi_i|_BP]=[\ker\pi_i|_B][P].
\end{equation}
Furthermore, since $P/I$
is also a principal comodule algebra, and $[B]=[P]^{co H/J}$
by Theorem~\ref{hogare},
we infer from Lemma~\ref{lplemma} 
 that $[B]\cap ([\ker\pi_i|_B][P])=[\ker\pi_i|_B]$ for any $i$.
Combining this with \eqref{kers} and remembering $B\cong [B]$ 
by Theorem~\ref{hogare}, we compute
\begin{equation*}
\bigcap_{i\in\{1,\cdots,N\}}\!\ker\bar{\pi_i}|_{[B]}
=\!\bigcap_{i\in\{1,\cdots,N\}}\!([B]\cap\ker\bar{\pi_i})
%=\bigcap_{i\in\{1,\cdots,N\}}[B]\cap ([\ker\pi_i|_B][P])
=\!\bigcap_{i\in\{1,\cdots,N\}}\![\ker\pi_i|_B]
= \Big[\bigcap_{i\in\{1,\cdots,N\}}\!\ker\pi_i|_B\Big]
=0.
\end{equation*}
%Furthermore, by \cite[Proposition~3.4]{HKMZ}, 
%the map from the lattice of ideals in $P/I$ that are also right 
%$H/J$-comodules to the lattice of all ideals in $[B]$ given by
% $K\mapsto [B]\cap K$ is a monomorphism of 
%lattices. Hence $\bigcap_{i\in\{1,\cdots,N\}}\ker\bar{\pi_i}=0$.
It also follows that the lattice generated by $\ker\bar{\pi_i}|_{[B]}$'s
is distributive because the lattice generated by $\ker\pi_i|_B$'s is distributive and 
$\ker\bar{\pi_i}|_{[B]}=[\ker\pi_i|_B]\cong \ker\pi_i|_B$ for all $i$. Hence 
$\{\bar\pi_i|_{[B]}\}_i$ is a covering of $[B]$ as needed.

% It also follows that the lattice generated by $\ker\bar{\pi_i}|_{[B]}$'s
% is distributive because it is isomorphic with the distributive
% lattice generated by all ideals $\ker\bar{\pi_i}|_{[B]}=[\ker\pi_i|_B]\cong \ker\pi_i|_B$.

Finally, to prove that the piecewise trivialisation of $P$ induces a
piecewise trivialisation of~$P/I$, it suffices to note that
the trivialisations (colinear algebra homomorphisms)
$\gamma_i$ descend to trvialisations of $\bar{P_i}$'s. Indeed, 
since for all $i$ we have 
$\gamma_i(J)\subseteq \pi_i(I)$, we conclude that there are maps
$\bar{\gamma}_i:H/J\ni [h]\mapsto [\gamma_i(h)]\in\bar{P}_i$. They
are colinear algebra homomorphisms, as needed. 
Summarising, we have shown that $P/I$ is a piecewise-trivial principal
$H/J$-comodule algebra, which ends the proof of one of the
implications asserted in Theorem~\ref{mainres}.

Conversely, now we want to prove that,
 if we can reduce a principal comodule algebra to a piecewise-trivial
principal comodule algebra,
then the comodule algebra we started from is piecewise-trivial
in a specific way. Our proof relies on the known fact that the
$H$-prolongation of a reduction of a principal $H$-comodule algebra is
isomorphic with this comodule algebra.
%It turns out that the cotensor product provides a required 
%construction. 
%The point about using this  construction is that the piecewise 
%principality will be 
%much easier to prove for the auxilliary space. The following lemma 
%provides a basic technical tool.

\begin{lemma}[\cite{qsng}]\label{cotensisomlem}
Let $P$ and $Q$ be principal comodule algebras over Hopf algebras
$H$ and $K$ respectively, let $g:H\rightarrow K$ be a morphism of Hopf 
algebras, and let $f:P\rightarrow Q$ be an algebra homomorphism that is 
colinear via~$g$. Assume also that $f$ restricted to $P^{co H}$  gives 
an isomorphism with $Q^{co K}$.
Then $P\cong Q\square^KH$ as comodule algebras.
%Let us denote by $\ell:H\rightarrow P\otimes P$, $h\mapsto\ls{h}\otimes 
%\rs{h}$
%a strong connection on $P$. Then the maps
%\begin{gather}
%F:P\longrightarrow Q\square^HK,\quad p\longmapsto\pi(p\sw{0})\otimes p
%\sw{1},\\
%G:Q\square^HK\longrightarrow P,\quad
%\sum_i q_i\otimes k_i\longmapsto\sum_iq_i\pi(\ls{k_i})\rs{k_i}
%\end{gather}
%are a pair of inverse left $B$-module, right $K$-comodule algebra 
%isomorphisms.
\end{lemma}

First, we consider  cotensor products with trivial comodule algebras. 
\begin{lemma}\label{cotenstr}
Let  $\pi:H\rightarrow\bar H$ be an epimorphism of Hopf algebras.
 Assume that $\bar P$ is a smash product $\bar H$-comodule algebra and 
$\bar\gamma:\bar H\rightarrow \bar P$ 
is its trivialisation (a colinear algebra
homomorphism). Denote $D:={}^{\co\bar H}H$ and $B:=\bar P^{\co\bar H}$. 
Then $\bar P\square^{\bar H}H$ is a smash product $H$-comodule
algebra and $\gamma:=((\bar\gamma\circ\pi)\ot\id)\circ\Delta:
H\rightarrow P\square^{\bar H}H$ 
is a trivialisation satisfying
$\gamma(k\sw{1})b\gamma(S(k\sw{2}))=0$ for all $b\in B$ and 
$k\in\ker\pi$. 
\end{lemma}
\begin{proof}
For any $b\in B$ and $k\in \ker\pi$, we obtain
\begin{align}
\gamma(k\sw{1})b\gamma(S(k\sw{2}))
&=\bar\gamma(\pi(k\sw{1}))b\bar\gamma(\pi(S(k\sw{4})))\otimes 
k\sw{2}S(k\sw{3})\\
&=\bar\gamma(\pi(k)\sw{1})b\bar\gamma(S(\pi(k)\sw{2}))\otimes 1\nonumber\\
&=0.\nonumber
\end{align}
Now, as $\gamma$ is clearly a colinear algebra homomorphism, we conclude the proof.
\end{proof}

Next, we prove a distributivity result for cotensor products
that will be useful 
in the proof of the subsequent lemma.
\begin{lemma}\label{idsumcot}
Let $\bar P$ be a principal $\bar H$-comodule algebra 
with $B:={\bar P}^{\co\bar  H}$, and let $\pi:H\rightarrow\bar H$ be an epimorphism  of Hopf 
algebras. Assume also that the antipode of $H$ is bijective.
% and  $H$ is a principal left $\bar H$-comodule algebra. 
Let $\bar{K}_1,\bar{K}_2\subseteq \bar P$ be ideals and right $\bar H$-
subcomodules in $\bar P$.
Then
\begin{equation}
\label{idsumcoteq}
\bar{K}_1\square^{\bar H}H+\bar{K}_2\square^{\bar H}H=
(\bar{K}_1+\bar{K}_2)\square^{\bar H}H.
\end{equation}
\end{lemma}
\begin{proof}
Let us denote $L_i:=B\cap \bar{K}_i$, $i=1,2$, for brevity. 
Using Lemma~\ref{lplemma}, we get 
\begin{equation}
\bar{K}_i=(\bar{K}_i\cap B)\bar P=L_i\bar{P},\quad i=1,2.
\end{equation}
Similarly, as $\bar{P}\square^{\bar H}H$ is a principal $H$-comodule
algebra
with $(\bar{P}\square^{\bar H}H)^{\co H}=B\otimes 1_H$, we can 
again apply Lemma~\ref{lplemma}  to obtain
\begin{equation*}
\bar{K}_i\square^{\bar H}H=\left(
(B\otimes 1_H)\cap \bar{K}_i\square^{\bar H}H
\right)(\bar{P}\square^{\bar H}H)
=((B\cap\bar{K}_i)\bar{P})\square^{\bar H}H
=L_i\bar{P}\square^{\bar H}H,\quad i=1,2.
\end{equation*}
Hence,
\begin{align}
\bar{K}_1\square^{\bar H}H+\bar{K}_2\square^{\bar H}H
&=L_1\bar{P}\square^{\bar H}H+L_2\bar{P}\square^{\bar H}H\nonumber\\
&=(L_1+L_2)\bar{P}\square^{\bar H}H\nonumber\\
&=((L_1+L_2)\bar P)\square^{\bar H}H\nonumber\\
&=(L_1\bar P+L_2\bar P)\square^{\bar H}H\nonumber\\
&=(K_1+K_2)\square^{\bar H}H,
\end{align}
as needed.
\end{proof}

Now we are ready to generalize  Lemma~\ref{cotenstr} from trivial 
comodule algebras to
 piecewise-trivial comodule algebras.
\begin{lemma}\label{almostinv}
Let $\bar P$ be a piecewise-trivial principal $\bar H$-comodule algebra 
with $B:={\bar P}^{\co\bar  H}$, let 
$\{\bar \pi_i:\bar P\rightarrow {\bar P}_i\}_{i\in\{1,\cdots,N\}}$ 
be a covering of $\bar P$, and let
 $\{{\bar\gamma}_i:\bar H\rightarrow {\bar P}_i\}_{i\in\{1,\cdots,N\}}$ 
be a family of 
trivialisations (colinear algebra homomorphisms). 
Assume also that $\pi:H\rightarrow\bar H$ is an epimorphism  of Hopf 
algebras,  the antipode of $H$ is bijective, and  
$H$ is a principal left $\bar H$-comodule algebra. 
Then $\bar P\square^{\bar H} H$ is a piecewise-trivial principal 
comodule algebra for the covering
\begin{equation}\label{family}
\{\bar\pi_i\square^{\bar H}\id_H:\bar P\square^{\bar H}H\longrightarrow 
{\bar P}_i\square^{\bar H}H\}_{i\in\{1,\cdots,N\}}\,,
\end{equation}
the maps
\begin{equation}\label{triv}
\{H\ni k\stackrel{\gamma_i}{\longmapsto}
\bar\gamma_i(\pi(k\sw{1}))\otimes k\sw{2}\in 
{\bar P}_i\square^{\bar H}H\}_{i\in\{1,\cdots,N\}}
\end{equation}
 are trivialisations
satisfying $\gamma_i(k\sw{1})b\gamma_i(S(k\sw{2}))=0$
 for all $b\in {\bar P}_i^{\co\bar H}\otimes 1$,
$k\in \ker \pi$, and
the associated transition functions  $T_{ij}$ (see~\eqref{trfundef}) 
fulfill
$T_{ij}(\ker \pi)=0$ for all $i,j\in\{1,\cdots,N\}$.
\end{lemma}
\begin{proof}
First note that since the principality of $H$ implies the coflatness of 
$H$ as a left 
$\bar H$-comodule \cite[Theorem~II.3.26]{qsng}, it follows that
 the maps $\bar\pi_i\otimes \id$ are all surjective.
Because $\{\bar\pi_i|_B\}_i$ is a covering of $B$,
it is immediate that 
$\{\left.\bar\pi_i\otimes\id_H\right|_{B\otimes 1_H}\}_i$ is a covering 
of $B\otimes 1_H=\left(\bar P\square^{\bar H} H\right)^{\co H}$.

%First we prove that 
%$\{\bar\pi_i\otimes\id_H:\bar P\square^{\bar H}H\rightarrow 
%{\bar P}_i\square^{\bar H}H\}_i$ is a weak covering.
%Since the principality of $H$ implies the coflatness of $H$ as a left 
%$\bar H$-comodule \cite[Theorem~II.3.26]{qsng}, it follows that
% the maps $\bar\pi_i\otimes \id$ are all surjective.
%On the other hand, the left exactness of the cotensor product functor
%implies that 
%$\bigcap_i\ker(\bar\pi_i\square^{\bar H}\id_H)=
%\bigcap_i\ker(\bar\pi_i)\square^{\bar H}H
%\subseteq\bigcap_i\ker\bar\pi_i\otimes H=0$.
%Furthermore, 
%as $(\bar P\square^{\bar H}H)^{\co H}=B\otimes 1$ and
% $(\bar P_i\square^{\bar H}H)^{\co H}={\bar P}_i^{\co\bar H}\otimes 1$
%for all $i$, the covering  
%$\{\bar\pi_i|_B: B\rightarrow{\bar P}_i^{\co\bar H}\}_i$
%induces the covering 
%$\{\bar\pi_i|_B\otimes\id
%: B\otimes 1\rightarrow{\bar P}_i^{\co\bar H}\otimes 1\}_i$.
%Hence $\bar P\square^{\bar H}H$ is a piecewise trivial principal
%comodule algebra. By \cite[Corollary~3.9]{HKMZ}, the weak covering
%\eqref{family} is a covering.

Next, from Lemma~\ref{cotenstr}, we conclude that all
the trivialisations \eqref{triv} satisfy
\begin{equation}
\gamma_i(k\sw{1})b\gamma_i(S(k\sw{2}))=0\quad\text{for all}\quad b\in {\bar P}_i^{\co\bar H}\otimes 1,
\;k\in \ker \pi.
\end{equation}
  Finally, we prove the desired property of the
associated transition functions.
The left exactness of the cotensor functor implies
that $\ker(\bar\pi_i\square^{\bar H}\id_H)
=(\ker\bar\pi_i)\square^{\bar H}H$. Combining this with 
Lemma~\ref{idsumcot} and the 
left coflatness of $H$ over $\bar H$,
we obtain the canonical isomorphism $\phi$
\begin{align}
(\bar P\square^{\bar H}H)/
(\ker(\bar\pi_i\square^{\bar H}\id_H)+
\ker(\bar\pi_j\square^{\bar H}\id_H))
&=(\bar P\square^{\bar H}H)/
(\ker(\bar\pi_i)\square^{\bar H}H+
\ker(\bar\pi_j)\square^{\bar H}H)\nonumber\\
&=(\bar P\square^{\bar H}H)/
(\ker\bar\pi_i+\ker\bar\pi_j)\square^{\bar H}H\nonumber\\
&\cong
({\bar P}/(\ker\bar\pi_i+\ker\bar\pi_j))\square^{\bar H}H.
\end{align}
Hence, we conclude
 that $\pi^i_j=\phi^{-1}\circ(\bar\pi^i_j\otimes\id_H)$
for all $i$ and $j$.
Therefore, 
we can write the transition functions (see \eqref{trfundef}) as
\begin{align}
T_{ij}(k)&=\pi^i_j(\gamma_i(k\sw{1}))\pi^j_i(\gamma_j(S(k\sw{2})))\nonumber\\
&=\phi^{-1}(\bar\pi^i_j(\bar\gamma_i(\pi(k\sw{1})))\bar\pi^j_i(\bar\gamma_j(\pi(S(k\sw{4}))))\otimes 
k\sw{2}S(k\sw{3}))\nonumber\\
&=\phi^{-1}(\bar\pi^i_j(\bar\gamma_i(\pi(k\sw{1})))\bar\pi^j_i(\bar\gamma_j(\pi(S(k\sw{2}))))\otimes 1_K).
\end{align}
Now the equality $T_{ij}(J)=0$ for any $i$ and $j$ follows from the fact
that  $J:=\ker\pi$ is a Hopf ideal.
\end{proof}

Summarising, it follows from Lemma~\ref{almostinv} and 
 Lemma~\ref{cotensisomlem} that, if a principal comodule algebra $P$
is reducible to a piecewise-trivial principal comodule algebra $\bar P$, 
then there exists a trivialisation of $P$ satisfying the two conditions
of the theorem.

\section{Noncommutative bundles over the Toeplitz deformation
of $\mathbb{R}P^2$}

\subsection{A quantum real projective space}

A new type of a noncommutative deformation of complex
projective spaces was constructed in~\cite{HKZ}. The construction is based on the
idea of covering a complex projective space by Cartesian powers of
closed discs (a compact restriction of the canonical affine covering).
Then discs are replaced by quantum discs \cite{kl93} given in terms of
the Toeplitz algebra~$\mathcal{T}$.
For real projective spaces $\mathbb{R}P^N$, $N-1\in\mathbb{N}$, a suitable compact restriction of the 
canonical affine covering is given by cubes $I^N$, where $I$ is the real unit disc, i.e.\ $I:=[-1,1]$. 
Now we replace $I^{2k}$ by $\mathcal{T}^{\otimes k}$ 
%and 
%$I^{2k+1}$ by $\mathcal{T}^{\otimes k}\otimes_{\mathrm{min}}C(I)$.
%Thus in the real case we are forced to consider the even and odd 
%dimension separately. 

Here we carry out the aforementioned construction for $N=2$. Hence,
the C*-algebra of our quantum $\mathbb{R}P^2$ will be a triple-pullback C*-algebra obtained from 
three Toeplitz algebras viewed
this time as the C*-algebras of quantum squares rather than quantum discs.
We consider the Toeplitz algebra~$\mathcal{T}$ as the universal
C*-algebra generated by an isometry $s$, and the symbol map
given by the assignment 
$\sigma\colon\mathcal{T}\ni s\mapsto \widetilde{u}\in C(S^1)$, where 
$\widetilde{u}$ is the unitary function
generating $C(S^1)$. Now we are ready ``to square the bounadry circle"
of the quantum disc with the help of the
following two maps
\begin{equation}
\ztwo\times I\ni(k,t)\stackrel{\delta_1}{\longmapsto} e^{i\pi(\frac{1}{4}kt+\frac{1}{2}k+\frac{3}{2})}\in 
S^1,\quad I\times\ztwo\ni(t,k)\stackrel{\delta_2}{\longmapsto} e^{i\pi(-\frac{1}{4}kt-\frac{1}{2}k+1)}\in 
S^1,
\label{deltadef}
\end{equation}
and their pullbacks
\begin{equation}\label{deltadef*}
\delta^*_1\colon C(S^1)\longrightarrow C(\ztwo)\otimes C(I),\quad
\delta^*_2\colon C(S^1)\longrightarrow  C(I)\otimes C(\ztwo).
\end{equation}
We will denote, for brevity, $\sigma_i:=\delta^*_i\circ\sigma$, 
$i=1,2$. Each
of the maps $\delta_i$ can be understood as a parametrisation of two appropriate quarters of $S^1$ as shown 
on the pictures below:
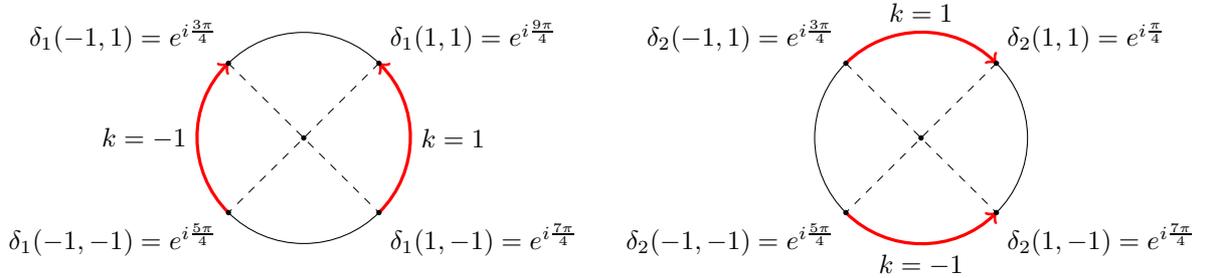
\begin{center}\label{pbelow}
\begin{tabular}{cc}
\begin{tikzpicture}[scale=0.7]
\fill (0,0) circle (0.5mm);
\draw (0,0) circle (2cm);
\draw[very thin, dashed] (45:2cm)--(225:2cm);
\draw[very thin, dashed] (135:2cm)--(315:2cm);
\draw[->,very thick,color=red] (225:2cm) arc (225:135:2cm);
\draw[->,very thick,color=red,rotate=-45] (0:2cm) arc (0:90:2cm); 
\fill (45:2cm) circle (0.5mm) (135:2cm) circle (0.5mm)(225:2cm) circle (0.5mm) (315:2cm) circle (0.5mm);
\draw (-2cm,0) node[anchor=east]{{\footnotesize $k=-1$}};
\draw (2cm,0) node[anchor=west]{{\footnotesize $k=1$}};
\draw (0,2cm) node[anchor=south]{{\footnotesize $\phantom{k=1}$}};
\draw (0,-2cm) node[anchor=north]{{\footnotesize $\phantom{k=-1}$}};
\draw (45:2cm) node[anchor=south west]{{\footnotesize $\delta_1(1,1)=e^{i\frac{9\pi}{4}}$}};
\draw (315:2cm) node[anchor=north west]{{\footnotesize $\delta_1(1,-1)=e^{i\frac{7\pi}{4}}$}};
\draw (135:2cm) node[anchor=south east]{{\footnotesize $\delta_1(-1,1)=e^{i\frac{3\pi}{4}}$}};
\draw (225:2cm) node[anchor=north east]{{\footnotesize $\delta_1(-1,-1)=e^{i\frac{5\pi}{4}}$}};
\end{tikzpicture} 

& \begin{tikzpicture}[scale=0.7]
\fill (0,0) circle (0.5mm);
\draw (0,0) circle (2cm);
\draw[very thin, dashed] (45:2cm)--(225:2cm);
\draw[very thin, dashed] (135:2cm)--(315:2cm);
\draw[->,very thick,color=red] (135:2cm) arc (135:45:2cm);
\draw[->,very thick,color=red] (225:2cm) arc (225:315:2cm); 
\fill (45:2cm) circle (0.5mm) (135:2cm) circle (0.5mm)(225:2cm) circle (0.5mm) (315:2cm) circle (0.5mm);
\draw (0,2cm) node[anchor=south]{{\footnotesize $k=1$}};
\draw (0,-2cm) node[anchor=north]{{\footnotesize $k=-1$}};
\draw (45:2cm) node[anchor=south west]{{\footnotesize $\delta_2(1,1)=e^{i\frac{\pi}{4}}$}};
\draw (315:2cm) node[anchor=north west]{{\footnotesize $\delta_2(1,-1)=e^{i\frac{7\pi}{4}}$}};
\draw (135:2cm) node[anchor=south east]{{\footnotesize $\delta_2(-1,1)=e^{i\frac{3\pi}{4}}$}};
\draw (225:2cm) node[anchor=north east]{{\footnotesize $\delta_2(-1,-1)=e^{i\frac{5\pi}{4}}$}};
\end{tikzpicture} \,. 
\end{tabular}
\end{center} 

We view $S^1$ and $I$ as $\ztwo$-spaces via multiplication by $\pm 1$.
Then $\ztwo\times I$ and $I\times \ztwo$ are $\ztwo$-spaces with the diagonal action. Accordingly, $C(I)$, $C(S^1)$,  
$C(\ztwo)\otimes C(I)$ and $C(I)\otimes C(\ztwo)$ are right 
$C(\ztwo)$-comodule algebras with coactions given by the pullbacks of respective $\ztwo$-actions.  
Denote by $u$
the generator $C(\ztwo)$ given by $u(\pm 1):=\pm 1$.
 Then the assignment $s\mapsto s\otimes u$ makes $\tplz$ a 
$C(\ztwo)$-comodule algebra. (This coaction corresponds to
the $\ztwo$-action given by $\alpha_{-1}^\tplz(s)=-s$.)
It is easy to verify that the maps $\delta_i$, $i=1,2$, are 
$\ztwo$-equivariant, so  their pullbacks
$\delta^*_i$'s  are right $C(\ztwo)$-comodule maps. Also, since
the symbol map $\sigma$ is a right $C(\ztwo)$-comodule map,  so are 
$\sigma_i$'s.

Now we are ready to define the C*-algebra $\qcp 2$
of our Toeplitz deformation
of the real projective plane. We take three copies of the
Toeplitz algebra $\tplz$, distinguish them by subscripts for clarity,
and write the building blocks of a triple pullback
diagram as follows:
\begin{gather}
\xymatrix{
\tplz_0 \ar[d]_{\sigma_1}& \tplz_1 \ar[d]^{\sigma_1}\\
C(\ztwo)\otimes C(I) & \ar[l]^{\Psi_{01}} C(\ztwo)\otimes C(I)
\,,}\quad
\xymatrix{
\tplz_0 \ar[d]_{\sigma_2}& \tplz_2 \ar[d]^{\sigma_1}\\
 C(I)\otimes C(\ztwo) & \ar[l]^{\Psi_{02}} C(\ztwo)\otimes C(I)
\,,}\nonumber\\
\xymatrix{
\tplz_1 \ar[d]_{\sigma_2}& \tplz_2 \ar[d]^{\sigma_2}\\
C(I)\otimes C(\ztwo) & \ar[l]^{\Psi_{12}} C(I)\otimes C(\ztwo)
\,.}
\end{gather}
Here the isomorphisms $\Psi_{ij}$ are given by formulae analogous
to the formulae used in \cite{HKZ}, that is: 
%\begin{subequations}
 \begin{align}
 C(\B{Z}_2)\otimes C(I)\ni u\otimes x&\stackrel{\Psi_{01}}{\longrightarrow} x\sw{1}u\otimes x\sw{0} \in C(\B{Z}_2)\otimes C(I),
\nonumber\\
C(\B{Z}_2)\otimes C(I)\ni u\otimes x&\stackrel{\Psi_{02}}{\longrightarrow} x\sw{0}\otimes x\sw{1}u \in C(I)\otimes C(\B{Z}_2),\\
C(I)\otimes C(\B{Z}_2) \ni x\otimes u&\stackrel{\Psi_{12}}{\longrightarrow}x\sw{0}\otimes x\sw{1}u\in   C(I)\otimes C(\B{Z}_2).\nonumber
 \end{align}
Putting all this together, we define the C*-algebra
\begin{align}
\qcp 2:=\bigl\{ 
(t_0,t_1,t_2)\in\tplz^3\;|\;\sigma_1(t_0)&=(\Psi_{01}\!\circ\!\sigma_1)(t_1),\\ \;\sigma_2(t_0)&=(\Psi_{02}\!
\circ\!\sigma_1)(t_2),\nonumber\\
\;\sigma_2(t_1)&=(\Psi_{12}\!\circ\!\sigma_2)(t_2)\bigr\}.\nonumber
\end{align}

\subsection{From $\proj R 2$ to a quantum 2-sphere}
\label{2.2}

The usual way of constructing real projective spaces is by taking
$\ztwo$-quotients of spheres. Here we reverse this procedure, i.e.\
we  treat
projective spaces as primary objects, and construct spheres from them. 
%projective spaces.
More precisely, since each cube covering a real projective space
is contractible, any principal bundle over such a cube must be trivial. Consequently, as the fiber of each 
principal bundle
$S^N\rightarrow\mathbb{R}P^N$, $N-1\in\mathbb{N}$, is $\ztwo$, we can assemble any sphere by 
appropriate glueing of pairs of cubes.
In particular, for $N=2$, we construct the topological 2-sphere
by assembling three pairs of squares to the boundary of a cube.

Our aim is to construct a quantum sphere $S^2_{\mathbb{R}\tplz}$ as a
 $\ztwo$-bundle over  $\proj R 2$. To this end, we take
$\tplz\otimes C(\ztwo)$ as basic  ingredients, and write
building blocks of a triple-pullback diagram
as follows:
\begin{gather}
\xymatrix{
\tplz_0\otimes C(\ztwo) \ar[d]_{\sigma_1\otimes\id}& \tplz_1\otimes C(\ztwo) \ar[d]^{\sigma_1\otimes\id}\\
C(\ztwo)\otimes C(I)\otimes C(\ztwo) & \ar[l]^{\widetilde\Phi_{01}} C(\ztwo)\otimes C(I)\otimes C(\ztwo)
\,,}\nonumber\\
~ \nonumber\\
\xymatrix{
\tplz_0\otimes C(\ztwo) \ar[d]_{\sigma_2\otimes\id}& \tplz_2\otimes C(\ztwo) \ar[d]^{\sigma_1\otimes\id}\\
 C(I)\otimes C(\ztwo)\otimes C(\ztwo) & \ar[l]^{\widetilde\Phi_{02}} C(\ztwo)\otimes C(I)\otimes C(\ztwo)
\,,}\nonumber\\
~ \nonumber\\
\xymatrix{
\tplz_1\otimes C(\ztwo) \ar[d]_{\sigma_2\otimes\id}& \tplz_2\otimes C(\ztwo) \ar[d]^{\sigma_2\otimes\id}\\
C(I)\otimes C(\ztwo)\otimes C(\ztwo) & \ar[l]^{\widetilde\Phi_{12}} C(I)\otimes C(\ztwo)\otimes C(\ztwo)
\,.}\label{gluecub}
\end{gather}
Here the isomorphisms  $\widetilde\Phi_{ij}$ are given by the formulae
(see \cite[eq.~9]{CM02})
\begin{equation}
\widetilde\Phi_{ij}(a\otimes b\otimes h):=\Psi_{ij}(a\otimes b)T_{ij}(h\sw{1})\otimes h\sw{2}\,,\quad 
i,j\in\{0,1,2\},\quad i<j,
\end{equation}  
for some  transitions functions $T_{ij}$ that will be determined later.
Now we can define our triple-pullback C*-algebra in the following
way:
\begin{align}
\widetilde{C(S^2_{\B R\tplz})}:=\bigl\{ 
(t_i\otimes u_i)_i\in
(\tplz\otimes C(\ztwo))^3\;|\;(\sigma_1\otimes\id)(t_0\otimes 
u_0)&=\bigl(\widetilde\Phi_{01}\circ(\sigma_1\otimes\id)\bigl)(t_1\otimes u_1),\nonumber\\ \;
(\sigma_2\otimes\id)(t_0\otimes u_0)&=\bigl(\widetilde\Phi_{02}\circ(\sigma_1\otimes\id)\bigr)(t_2\otimes 
u_2),\nonumber\\
\;(\sigma_2\otimes\id\bigr)(t_1\otimes u_1)&=\bigl(\widetilde\Phi_{12}\circ(\sigma_2\otimes\id))(t_2\otimes 
u_2)\bigr\}.
\end{align}

If we consider the natural $\ztwo$-actions
on the rightmost tensorands of the components of 
$\widetilde{C(S^2_{\B R\tplz})}$, 
all maps in the diagram \eqref{gluecub}
 are 
$\ztwo$-equivariant C*-homomorphisms. 
%On the  other hand, we can also consider the diagonal $\ztwo$-action. 
Thus, we obtain a $\ztwo$-action on $\widetilde{C(S^2_{\B R\tplz})}$
such that its fixed-point subalgebra satisfies
\[\label{equal}
\widetilde{C(S^2_{\B R\tplz})}^{\ztwo}=C(\proj R2)\otimes\B C.
\]
Trading the above action for coaction, we can view the components
of $\qcube$ as trivial $C(\ztwo)$-comodule algebras. 
However, to see that the quantum real projective space $\proj R2$ corresponds to the quotient of 
$S^2_{\B R\tplz}$ by the antipodal
 $\ztwo$-action, we need to  transform $\widetilde{C(S^2_{\B R\tplz})}$
 into an appropraite isomorphic $\ztwo$-C*-algebra. To this end, 
we need to gauge the aforementioned  $\ztwo$-actions on components
to the diagonal $\ztwo$-actions thereon.
We transform the former  into the latter 
by conjugating
all maps of \eqref{gluecub} by
the gauge transformation of the form 
\begin{equation}
\label{gbdef}
g_B:B\otimes C(\ztwo)\ni b\otimes h\longmapsto b\sw{0}\otimes b\sw{1}h\in B\otimes C(\ztwo).
\end{equation}
To define the diagonal action on the right-hand side,
 we view $B$ as  one of the following $\ztwo$-C*-algebras: 
 $C(\ztwo)\otimes C(I)$ with
the diagonal antipodal action,
$C(I)\otimes C(\ztwo)$ again with the diagonal antipodal action,
or $\tplz$ with the $\ztwo$-action given by $\alpha_{-1}^\tplz(s)=-s$.

For brevity, in what follows we will omit 
 the subscript distinguishing various maps $g$ whenever it is implied by the context.
 To compute the result of conjugation of morphisms in \eqref{gluecub} with maps $g$, first we note that 
$g=g^{-1}$ because the antipode of 
$C(\ztwo)$  is equal to the identity function. 
Next, it is immediate to verify that, due to the right
$C(\ztwo)$-colinearity of $\sigma_i$'s,  we obtain
 \begin{equation}
 g\circ(\sigma_1\otimes\id)\circ g=(\sigma_1\otimes\id)\quad\text{and}\quad g\circ(\sigma_2\otimes\id)\circ 
g=(\sigma_2\otimes\id).
 \end{equation}
Furthermore, let us define $\Phi_{ij}:=g\circ\tilde\Phi_{ij}\circ g$,
and compute:
\begin{align}\label{gluexpl}
\Phi_{01}(h\otimes p\otimes k)&=(g\circ\tilde\Phi_{01}\circ g)(h\otimes p\otimes k)\nonumber\\
&=(g\circ\tilde\Phi_{01})(h\sw 1 \otimes p\sw 0 \otimes  h\sw 2p\sw 1 k)\nonumber\\
&=g\bigl(
\Psi_{01}(h\sw 1 \otimes p\sw 0)T_{01}( h\sw{2}p\sw{1}k\sw{1})\otimes h\sw{3}p\sw{2}k\sw{2}
\bigr)\nonumber\\
&=g\bigl(
(h\sw{1}p\sw{1} \otimes p\sw{0})T_{01}( h\sw{2}p\sw{2}k\sw{1})\otimes h\sw{3}p\sw{3}k\sw{2}
\bigr)\nonumber\\
&=(h\sw{1}p\sw{2}\otimes p\sw{0})T_{01}(h\sw{3}p\sw{4}k\sw{1})\sw{0}\otimes h\sw{2}p\sw{3}p\sw{1}T_{01}(h\sw{3}p\sw{4}k\sw{1})\sw{1}
h\sw{4}p\sw{5}k\sw{2}\nonumber\\
&=(h\sw{1}p\sw{1}\otimes p\sw{0})T_{01}(h\sw{2}p\sw{5}k\sw{1})\sw{0}\otimes h\sw{3}h\sw{4}
p\sw{2}p\sw{3}T_{01}(h\sw{2}p\sw{5}k\sw{1})\sw{1}
p\sw{4}k\sw{2}\nonumber\\
&=(h\sw{1}p\sw{1}\otimes p\sw{0})T_{01}(h\sw{2}p\sw{3}k\sw{1})\sw{0}\otimes 
T_{01}(h\sw{2}p\sw{3}k\sw{1})\sw{1}p\sw{2}k\sw{2}\,.
\end{align}
The penultimate line above follows from the commutativity and cocomutativity of $C(\ztwo)$. The last equality 
is a consequence of
$h\sw{1}h\sw{2}=\varepsilon(h)$ for all $h\in C(\ztwo)$. The computations for $\Phi_{02}$ and $\Phi_{12}$ 
are similar.

Finally, we determine the transition functions $T_{ij}$.
% so that $C(S^2_{\B R\tplz})$ is indeed a noncommutative
% deformation of~$C(S^2)$.
To this end,  we observe that $\Phi_{01}$ is the pullback
of the following map:
\begin{align}
\ztwo\times I\times\ztwo&\stackrel{\Psi_{01}^*}{\longrightarrow}\ztwo\times I\times\ztwo,\\
(a,t,c)&\longmapsto (af_{01}(ac,tc),\,atcf_{01}(ac,ct),\,cf_{01}(ac,tc)).\nonumber
\end{align}
Here $f_{01}:\ztwo\times I\to \ztwo$ is  the map whose pullback
is $T_{01}$. Much in the same way, we note that $\Phi_{02}$ is
the pullback of
\begin{align}
\ztwo\times I\times\ztwo&\stackrel{\Psi_{02}^*}{\longrightarrow}I\times\ztwo\times \ztwo,\\
(a,t,c)&\longmapsto (atcf_{02}(ac,tc),\,af_{02}(ac,tc),\,cf_{02}(ac,tc)),\nonumber
\end{align}
and $\Phi_{12}$ is the pullback of
\begin{align}
I\times\ztwo\times \ztwo&\stackrel{\Psi_{12}^*}{\longrightarrow}I\times\ztwo\times \ztwo,\\
(t,a,c)&\longmapsto (atcf_{12}(ac,tc),\,af_{12}(ac,tc),\,cf_{12}(ac,tc)).\nonumber
\end{align}
The continuity of $f_{ij}$'s implies that they are independent
of their continuous argument. Hence, we have to choose only between four possible functions: $\id$, $-\id$, 
$1$, $-1$.
We choose $f_{ij}=\id$ for all $i,j\in\{0,1,2\}$, $i<j$.
Thus, we obtain
\begin{align}\label{phiijdef}
\Phi_{01}(h\otimes p\otimes k)&=k\otimes p\otimes h,\nonumber\\
\Phi_{02}(h\otimes p\otimes k)&=p\otimes k\otimes h,\nonumber\\
\Phi_{12}( p\otimes h\otimes k)&=p\otimes k\otimes h.
\end{align} 

We are now ready to define the following triple-pullback C*-algebra:
\begin{align}\label{qcube}
{C(S^2_{\B R\tplz})}:=\bigl\{ 
(t_i\otimes u_i)_i\in
(\tplz\otimes C(\ztwo))^3\;|\;(\sigma_1\otimes\id)(t_0\otimes 
u_0)&=\bigl(\Phi_{01}\circ(\sigma_1\otimes\id)\bigl)(t_1\otimes u_1),\nonumber\\ \;(\sigma_2\otimes\id)
(t_0\otimes u_0)&=\bigl(\Phi_{02}\circ(\sigma_1\otimes\id)\bigr)(t_2\otimes u_2),\nonumber\\
\;(\sigma_2\otimes\id\bigr)(t_1\otimes u_1)&=\bigl(\Phi_{12}\circ(\sigma_2\otimes\id))(t_2\otimes u_2)\bigr\}.
\end{align}
Since the diagonal and the rightmost $\ztwo$-actions on the components
of $C(S^2_{\B R\tplz})$ and $\widetilde{C(S^2_{\B R\tplz})}$, respectively, are intertwined by
 C*-isomorphisms, we conclude that they
are isomorphic as $\ztwo$-C*-algebras. Consequently,
 their invariant subalgebras
are naturally isomorphic. Combining this with \eqref{equal}, we
obtain an isomorphism of C*-algebras:
\[\label{cong}
C(S^2_{\B R\tplz})^{\ztwo}\cong C(\proj R2).
\]

One can check that replacing in the foregoing construction of
$C(S^2_{\B R\tplz})$ the Toeplitz algebra $\mathcal{T}$ by
the algebra $C(D)$ of continuous functions on the unit disc,
yields a C*-algebra isomorphic with $C(S^2)$. 
Also, 
the $\ztwo$-action on $C(S^2_{\B R\tplz})$, which is
given by the diagonal action on each component, becomes precisely the
the pullback of the diagonal (antipodal) action on~$S^2$.
The isomorphism is
 given by rounding the boundary of a cube to the unit sphere.
Indeed, using the notation
\[
\mathcal{T}_{i,j}:=(\id\otimes\ev_j)(\mathcal{T}_i\otimes C(\ztwo)),\quad 
E_{1,i}:=(\ev_i\otimes\id)\circ\sigma_1,\quad
E_{2,i}:=(\id\otimes\ev_i)\circ\sigma_2,\label{nota}
\]
allows us to verify it with the help of the following picture:
{\newcommand{\dd}{3.5cm}
\newcommand{\ddh}{1.75cm}
\newcommand{\ead}{3mm}
\newcommand{\ebd}{3.2cm}
\newcommand{\eada}{0.9cm}
\newcommand{\ebda}{2.6cm}
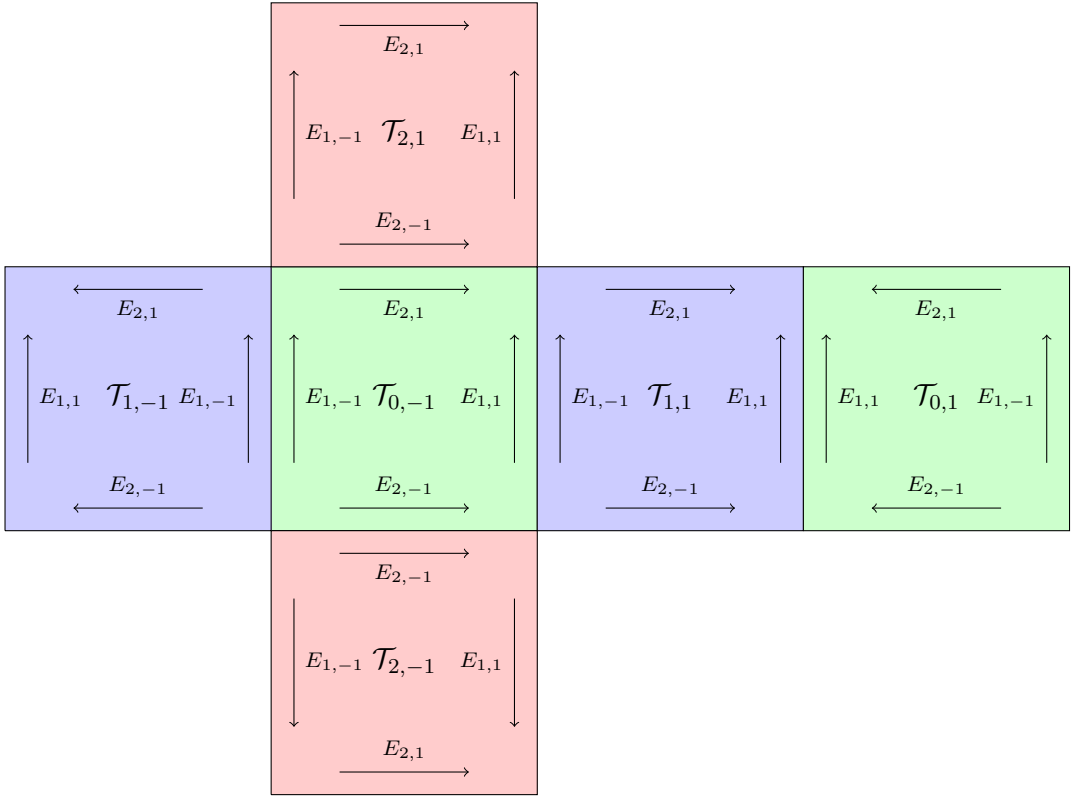
\begin{center}
\begin{tikzpicture}
\draw[fill=green!20!white,draw=black] (0,0)--(\dd,0)--(\dd,\dd)--(0,\dd)--(0,0);
\draw[fill=blue!20!white,draw=black,xshift=-\dd] (0,0)--(\dd,0)--(\dd,\dd)--(0,\dd)--(0,0);
\draw[fill=blue!20!white,draw=black,xshift=\dd] (0,0)--(\dd,0)--(\dd,\dd)--(0,\dd)--(0,0);
\draw[fill=green!20!white,draw=black,xshift=\dd][xshift=\dd] (0,0)--(\dd,0)--(\dd,\dd)--(0,\dd)--(0,0);
\draw[fill=red!20!white,draw=black,yshift=-\dd] (0,0)--(\dd,0)--(\dd,\dd)--(0,\dd)--(0,0);
\draw[fill=red!20!white,draw=black,yshift=\dd] (0,0)--(\dd,0)--(\dd,\dd)--(0,\dd)--(0,0);
\draw (1.75cm,1.75cm) node{\small $\tplz_{0,-1}$};
\draw[xshift=-\dd] (\ddh,\ddh) node{\small $\tplz_{1,-1}$};
\draw[xshift=\dd] (\ddh,\ddh) node{\small $\tplz_{1,1}$};
\draw[xshift=\dd][xshift=\dd] (\ddh,\ddh) node{\small $\tplz_{0,1}$};
\draw[yshift=-\dd] (\ddh,\ddh) node{\small $\tplz_{2,-1}$};
\draw[yshift=\dd] (\ddh,\ddh) node{\small $\tplz_{2,1}$};

\draw[->] (\ead,\eada)--(\ead,\ebda);
\draw[->] (\eada,\ebd)--(\ebda,\ebd);
\draw[->] (\eada,\ead)--(\ebda,\ead);
\draw[->] (\ebd,\eada)--(\ebd,\ebda);

\draw[->,xshift=-\dd] (\ead,\eada)--(\ead,\ebda);
\draw[<-,xshift=-\dd] (\eada,\ebd)--(\ebda,\ebd);
\draw[<-,xshift=-\dd] (\eada,\ead)--(\ebda,\ead);
\draw[->,xshift=-\dd] (\ebd,\eada)--(\ebd,\ebda);

\draw[->,xshift=\dd] (\ead,\eada)--(\ead,\ebda);
\draw[->,xshift=\dd] (\eada,\ebd)--(\ebda,\ebd);
\draw[->,xshift=\dd] (\eada,\ead)--(\ebda,\ead);
\draw[->,xshift=\dd] (\ebd,\eada)--(\ebd,\ebda);

\draw[->,xshift=\dd][xshift=\dd] (\ead,\eada)--(\ead,\ebda);
\draw[<-,xshift=\dd][xshift=\dd] (\eada,\ebd)--(\ebda,\ebd);
\draw[<-,xshift=\dd][xshift=\dd] (\eada,\ead)--(\ebda,\ead);
\draw[->,xshift=\dd][xshift=\dd] (\ebd,\eada)--(\ebd,\ebda);

\draw[->,yshift=\dd] (\ead,\eada)--(\ead,\ebda);
\draw[->,yshift=\dd] (\eada,\ebd)--(\ebda,\ebd);
\draw[->,yshift=\dd] (\eada,\ead)--(\ebda,\ead);
\draw[->,yshift=\dd] (\ebd,\eada)--(\ebd,\ebda);

\draw[<-,yshift=-\dd] (\ead,\eada)--(\ead,\ebda);
\draw[->,yshift=-\dd] (\eada,\ebd)--(\ebda,\ebd);
\draw[->,yshift=-\dd] (\eada,\ead)--(\ebda,\ead);
\draw[<-,yshift=-\dd] (\ebd,\eada)--(\ebd,\ebda);

\draw (\ead,\ddh) node[anchor=west]{\tiny $E_{1,-1}$};
\draw (\ebd,\ddh) node[anchor=east]{\tiny $E_{1,1}$};
\draw (\ddh,\ead) node[anchor=south]{\tiny $E_{2,-1}$};
\draw (\ddh,\ebd) node[anchor=north]{\tiny $E_{2,1}$};

\draw[xshift=-\dd] (\ead,\ddh) node[anchor=west]{\tiny $E_{1,1}$};
\draw[xshift=-\dd] (\ebd,\ddh) node[anchor=east]{\tiny $E_{1,-1}$};
\draw[xshift=-\dd] (\ddh,\ead) node[anchor=south]{\tiny $E_{2,-1}$};
\draw[xshift=-\dd] (\ddh,\ebd) node[anchor=north]{\tiny $E_{2,1}$};

\draw[xshift=\dd] (\ead,\ddh) node[anchor=west]{\tiny $E_{1,-1}$};
\draw[xshift=\dd] (\ebd,\ddh) node[anchor=east]{\tiny $E_{1,1}$};
\draw[xshift=\dd] (\ddh,\ead) node[anchor=south]{\tiny $E_{2,-1}$};
\draw[xshift=\dd] (\ddh,\ebd) node[anchor=north]{\tiny $E_{2,1}$};

\draw[xshift=\dd][xshift=\dd] (\ead,\ddh) node[anchor=west]{\tiny $E_{1,1}$};
\draw[xshift=\dd][xshift=\dd] (\ebd,\ddh) node[anchor=east]{\tiny $E_{1,-1}$};
\draw[xshift=\dd][xshift=\dd] (\ddh,\ead) node[anchor=south]{\tiny $E_{2,-1}$};
\draw[xshift=\dd][xshift=\dd] (\ddh,\ebd) node[anchor=north]{\tiny $E_{2,1}$};

\draw[yshift=\dd] (\ead,\ddh) node[anchor=west]{\tiny $E_{1,-1}$};
\draw[yshift=\dd] (\ebd,\ddh) node[anchor=east]{\tiny $E_{1,1}$};
\draw[yshift=\dd] (\ddh,\ead) node[anchor=south]{\tiny $E_{2,-1}$};
\draw[yshift=\dd] (\ddh,\ebd) node[anchor=north]{\tiny $E_{2,1}$};

\draw[yshift=-\dd] (\ead,\ddh) node[anchor=west]{\tiny $E_{1,-1}$};
\draw[yshift=-\dd] (\ebd,\ddh) node[anchor=east]{\tiny $E_{1,1}$};
\draw[yshift=-\dd] (\ddh,\ead) node[anchor=south]{\tiny $E_{2,1}$};
\draw[yshift=-\dd] (\ddh,\ebd) node[anchor=north]{\tiny $E_{2,-1}$};
\end{tikzpicture}
\end{center}}
\begin{remark}
It is worth mentioning that the choice
$f_{ij}=1$ for all $i,j\in\{0,1,2\}$, $i<j$, would yield
the C*-algebra $\qcp 2\otimes C(\ztwo)$.
\end{remark}

\subsection{The quantum $\ztwo$-principal bundle 
$S^2_{\B R\tplz}\rightarrow\proj R 2$}
\label{2.3}

To prove that the  $C(\ztwo)$-comodule algebra $\qcube$
constructed above as a triple-pullback comodule algebra
is principal, first we need to show that all restrictions
$\qcube\to (\mathcal{T}\otimes C(\ztwo))_i$ of the canonical
surjections remain surjective. 
A sufficient condition for the aforementioned surjectivity is given 
in the following technical proposition.

\begin{proposition}\ \cite[Proposition~9]{CM00} \label{MattProp}
Let us denote, for brevity, $\underline{N}:=\{0,\ldots,N\}$.
 Let $\{B_i\}_{i\in{\underline{N}}}$ and
  $\{B_{ij}\}_{i,j\in{\underline{N}},\,i\neq j}$ be two families of
algebras
  such that
  $B_{ij}=B_{ji}$\,, and let $\{\pi^i_j:B_i\rightarrow B_{ij}\}_{ij}$
be a family of surjective algebra maps whose kernels generate a distributive lattice of ideals. Also,
let $\pi_i:B\rightarrow B_i$, $i\in{\underline{N}}$,
  be the restrictions to
$$
B:=\{(b_i)_i\in\mbox{$\prod_{i\in{\underline{N}}}$}B_i\;|\;\pi^i_j(b_i)=\pi^j_i(b_j),\;\forall\,
i,j\in{\underline{N}},\;i\neq j\}
$$
 of the canonical projections.  Assume that, for all triples
of distinct indices $i,j,k\in{\underline{N}}$, the following
conditions hold:
\begin{enumerate}
%\item for all $I\subseteq\underbar{N}$ and $j,k\in\underbar{N}\setminus I$ %we have
%\begin{equation*}
%\bigcap_{i\in I}(\ker\pi^k_i+\ker\pi^k_j)=\left(\bigcap_{i\in I}\ker\pi^k_i\right)+\ker\pi^k_j\,;
%\end{equation*}
\item $\pi^i_j(\ker\pi^i_k)=\pi^j_i(\ker\pi^j_k)$;
\vspace{1mm}
\item the isomorphisms $\pi^{ij}_k\colon
  B_i/(\ker\pi^i_j+\ker\pi^i_k)\longrightarrow
  B_{ij}/\pi^i_j(\ker\pi^i_k)$ defined as
  \begin{align*}
    &b_i+\ker\pi^i_j+\ker\pi^i_k\longmapsto \pi^i_j(b_i)+\pi^i_j(\ker\pi^i_k)
  \\
  \text{ satisfy }\quad  &(\pi^{ik}_j)^{-1}\circ\pi^{ki}_j=(\pi^{ij}_k)^{-1}\circ\pi^{ji}_k\circ(\pi^{jk}_i)^{-1}\circ\pi^{kj}_i.
  \end{align*}
\end{enumerate}\vspace*{-5mm}
\begin{align*}
\text{Then, }\;&\mbox{\Large$\forall\;$}(b_i)_{i\in I}\in\mbox{$\prod_{i\in I}$}\;
B_i,\;  I\subseteq{\underline{N}},
\text{ such that }
\pi^i_j(b_i)=\pi^j_i(b_j),\, \forall\; i,j\in I,\,i\neq j,
\\
&\mbox{\Large$\exists\;$}
(c_i)_{i\in{\underline{N}}}\in\mbox{$\prod_{i\in{\underline{N}}}$}\; B_i : \:
\pi^i_j(c_i)=\pi^j_i(c_j),\,\forall\; i,j\in{\underline{N}},\,i\neq j, \text{ and }
c_i=b_i, \forall\; i\in I.\phantom{WWWWWWWWWWWW}
\end{align*}
\end{proposition}

Our task  now is to check that our multipullback construction of
$\qcube$ satisfies the assumptions of Proposition~\ref{MattProp}.
The distributivity condition is automatically satisfied because 
here we work with C*-ideals, and the lattices of C*-ideals
are always distributive. 
We begin by defining certain auxiliary maps 
%that will turn out to be instrumental in constructing  elements in the preimages
%of $\delta^*_i$'s (see \eqref{deltadef*}) and other maps. 
%Such example elements will be necessary for some of the proofs in what follows.
%Let 
$\hat\phi_1,\hat\phi_2\in C(S^1)$ by the formulae:
\begin{equation}
\hat\phi_1(e^{i\theta}):=
\left\{
\begin{array}{ccl}
2-\frac{4}{\pi}\theta & \text{if} &\theta\in[\frac{\pi}{4},\frac{3\pi}{4}]\\
&&\\
-1 & \text{if} & \theta\in[\frac{3\pi}{4},\frac{5\pi}{4}]\\
&&\\
\frac{4}{\pi}\theta-6 &\text{if} & \theta\in[\frac{5\pi}{4},\frac{7\pi}{4}]\\
&&\\
1 &\text{if} & \theta\in[\frac{7\pi}{4},\frac{9\pi}{4}]
\end{array}
\right.,\quad
\hat\phi_2(e^{i\theta}):=
\left\{
\begin{array}{ccl}
1 & \text{if} &\theta\in[\frac{\pi}{4},\frac{3\pi}{4}]\\
&&\\
4-\frac{4}{\pi}\theta & \text{if} & \theta\in[\frac{3\pi}{4},\frac{5\pi}{4}]\\
&&\\
-1 &\text{if} & \theta\in[\frac{5\pi}{4},\frac{7\pi}{4}]\\
&&\\
\frac{4}{\pi}\theta-8 &\text{if} & \theta\in[\frac{7\pi}{4},\frac{9\pi}{4}]
\end{array}
\right..\label{phidefeq}
\end{equation}
One immediately sees that
\[\label{line2}
\hat\phi_1,\hat\phi_2:S^1\longrightarrow [-1,1],\quad
\hat\phi_1(-z)=-\hat\phi_1(z),\quad \hat\phi_2(-z)=-\hat\phi_2(z).
\]
Next, let us denote by $\imath_I\in C(I)$ the inclusion given by
 $\imath_I(t):=t$, where
$t\in I:=[-1,1]$. Recalling that $u$ is the generator of $C(\ztwo)$
given by $u(\pm 1) = \pm 1$, and remembering \eqref{deltadef},
one easily verifies the following properties of $\hat\phi_i$'s:
\[
\hat\phi_1\circ\delta_1=u\otimes 1_{C(I)},\quad
\hat\phi_2\circ\delta_2=1_{C(I)}\otimes u,\quad
\hat\phi_1\circ\delta_2=\imath_I\otimes 1_{C(\ztwo)},\quad
\hat\phi_2\circ\delta_1=1_{C(\ztwo)}\otimes \imath_I.\label{phiprop}
\]

Furthermore, we need to define
unital and right $C(\ztwo)$-colinear splittings 
\begin{equation}
\hat\omega_1:C(\ztwo)\otimes C(I)\longrightarrow C(S^1),\qquad
\hat\omega_2:C(I)\otimes C(\ztwo)\longrightarrow C(S^1),
\end{equation}
of $\delta_1^*$ and $\delta_2^*$, respectively. To this end,
take any $h\in C(I)$ and set
\begin{align}
\label{omegadef}
\hat\omega_1(1\otimes h):=h\circ\hat\phi_2,&\quad
\hat\omega_1(u\otimes h):=\hat\phi_1\cdot (h\circ\hat\phi_2),\nonumber\\
\hat\omega_2(h\otimes 1):=h\circ\hat\phi_1,&\quad
\hat\omega_2(h\otimes u):=\hat\phi_2\cdot (h\circ\hat\phi_1).
\end{align}
Here $\cdot$ stands for the pointwise multiplication in $C(S^1)$.
The right colinearity of $\hat\omega_i$'s follows immediately from  \eqref{line2}, and it is straightforward to 
check that $\hat\omega_i$'s are
splittings, i.e.\ \mbox{$\delta_1^*\circ\hat\omega_1=\id$}, $\delta_2^*\circ\hat\omega_2=\id$. Indeed,
take any $h\in C(I)$ and use \eqref{phiprop} to compute: 
\begin{align}
(\delta_1^*\circ\hat\omega_1)&(1_{C(\ztwo)}\otimes h)=h\circ\hat\phi_2\circ\delta_1=
h\circ(1_{C(\ztwo)}\otimes \imath_I)=(1_{C(\ztwo)}\otimes h),\nonumber\\
(\delta_1^*\circ\hat\omega_1)&(u\otimes h)=(\hat\phi_1\circ\delta_1)\cdot (h\circ\hat\phi_2\circ\delta_1)
=(u\otimes 1_{C(I)})\cdot (1_{C(\ztwo)}\otimes h)=u\otimes h.
\end{align}
The case of $\delta_2^*\circ\hat\omega_2$ is analogous.

To prove certain additional properties of $\omega_1$ and $\omega_2$,
let us denote by $\imath_\ztwo:\ztwo\rightarrow I$ the inclusion map, 
by $\imath_\ztwo^*:C(I)\rightarrow C(\ztwo)$ its pullback,  and 
 by $\imath^\ztwo_*:C(\ztwo)\rightarrow C(I)$ the right 
$C(\ztwo)$-colinear splitting of $\imath^*_\ztwo$ defined by the formula
\begin{equation}
\imath^\ztwo_*(1_{C(\ztwo)})=1_{C(I)},\quad \imath^\ztwo_*(u)=\imath_I.
\end{equation}
Now we are ready to verify that
\begin{equation}
\label{omegaprop}
\delta_2^*\circ\hat\omega_1=\imath^\ztwo_*\otimes\imath^*_\ztwo,\quad
\delta_1^*\circ\hat\omega_2=\imath_\ztwo^*\otimes\imath_*^\ztwo.
\end{equation}
To this end, again take
any $h\in C(I)$ and use \eqref{phiprop} to compute
\begin{align}
(\delta_2^*\circ\hat\omega_1)(1_{C(\ztwo)}\otimes h)
&=h\circ\hat\phi_2\circ\delta_2=h\circ(1_{C(I)}\otimes u)\nonumber\\
&=\imath^\ztwo_*(1_{C(\ztwo)})\otimes\imath^*_\ztwo(h),\nonumber\\
(\delta_2^*\circ\hat\omega_1)(u\otimes h)&=(\hat\phi_1\circ\delta_2)\cdot(h\circ\hat\phi_2\circ\delta_2)=(\imath_I\otimes 1_{C(\ztwo)})\cdot\bigl(1_{C(I)}\otimes\imath^*_\ztwo(h)\bigr)\nonumber\\
&=\imath^\ztwo_*(u)\otimes\imath^*_\ztwo(h).
\end{align}
The proof of the second  equality in \eqref{omegaprop} is similar.

As the last  prerequisite to check that the assumptions of 
Proposition~\ref{MattProp} are satisfied, we note the following property of the kernels of $\delta^*_i$'s:
\begin{equation}
\label{deltakerprop}
\delta^*_1(\ker\delta^*_2)=C(\ztwo)\otimes\ker\imath^*_\ztwo,\quad
\delta^*_2(\ker\delta^*_1)=\ker\imath^*_\ztwo\otimes C(\ztwo).
\end{equation} 
Recalling that $\sigma_i:=\delta_i^*\circ\sigma$, $i=1,2$,
we can combine
 $\ker\sigma_i=\sigma^{-1}(\ker\delta_i^*)$, $i=1,2$, with 
\eqref{deltakerprop} to obtain
\begin{equation}
\label{sigmakerprop}
\sigma_1(\ker\sigma_2)=C(\ztwo)\otimes\ker\imath^*_\ztwo,\quad
\sigma_2(\ker\sigma_1)=\ker\imath^*_\ztwo\otimes C(\ztwo).
\end{equation}

Let us now instantiate Condition~(1)  of Proposition~\ref{MattProp}
for $N=2$:
\begin{equation}
\label{Mattjed}
\pi^0_1(\ker\pi^0_2)=\pi^1_0(\ker\pi^1_2),\quad
\pi^0_2(\ker\pi^0_1)=\pi^2_0(\ker\pi^2_1),\quad
\pi^1_2(\ker\pi^1_0)=\pi^2_1(\ker\pi^2_0),
\end{equation} 
where
\begin{gather}
\pi^0_1:=\sigma_1\otimes\id,\quad \pi^1_0:=\Phi_{01}\circ(\sigma_1\otimes\id),\quad
\pi^0_2:=\sigma_2\otimes\id,\quad \pi^2_0:=\Phi_{02}\circ(\sigma_1\otimes\id),\nonumber\\ 
\pi^1_2:=\sigma_2\otimes\id,\quad \pi^2_1:=\Phi_{12}\circ(\sigma_2\otimes\id),\label{jedi}
\end{gather}
and $\Phi_{ij}$'s are given by~\eqref{phiijdef}. Taking advantage
of \eqref{sigmakerprop}, we check the first equality of~\eqref{Mattjed}:
\begin{align}
\pi^1_0(\ker\pi^1_2)
&=\Phi_{01}\bigl((\sigma_1\otimes\id)(\ker(\sigma_2\otimes\id))\bigr)
\nonumber\\
&=\Phi_{01}\bigl(\sigma_1(\ker\sigma_2)\otimes C(\ztwo)\bigr)
\nonumber\\
&=\Phi_{01}\bigl(C(\ztwo)\otimes\ker\imath^*_\ztwo\otimes C(\ztwo)\bigr)
\nonumber\\
&=C(\ztwo)\otimes\ker\imath^*_\ztwo\otimes C(\ztwo)
\nonumber\\
&=\sigma_1(\ker\sigma_2)\otimes C(\ztwo)
\nonumber\\
&=(\sigma_1\otimes\id)(\ker(\sigma_2\otimes\id))
\nonumber\\
&=\pi^0_1(\ker\pi^0_2).
\end{align}
Observe that the remaining equalities of \eqref{Mattjed}
 can be verified in the same way.

Condition~(2) of Proposition~\ref{MattProp} for $N=2$ gives
us 6 equalities of the form 
$\phi^{ik}_j=\phi^{ij}_k\circ\phi^{jk}_i$, 
where $\phi^{ij}_k:=(\pi^{ij}_k)^{-1}\circ\pi^{ji}_k$.
Since $(\phi^{ij}_k)^{-1}=\phi^{ji}_k$, these 6 equalities
are pairwise equivalent. Thus, 
it suffices to show only one of them. We choose the equality
$\phi^{02}_1=\phi^{01}_2\circ\phi^{12}_0$ and write it as
\begin{equation}
\label{lastmatcond}
\pi^{01}_2\circ(\pi^{02}_1)^{-1}\circ\pi^{20}_1=\pi^{10}_2\circ(\pi^{12}_0)^{-1}\circ\pi^{21}_0.
\end{equation}
Next, denote by 
\begin{equation}
[\cdot]^i_{jk}:B_i\rightarrow B_i/(\ker\pi^i_j+\ker\pi^i_k),\quad
[\cdot]^{ij}_k:B_{ij}\rightarrow B_{ij}/\pi^i_j(\ker\pi^i_k), 
\end{equation}
the natural epimorphisms.
Using the splittings of $\delta^*_i$'s defined by \eqref{omegadef}
and remembering~\eqref{jedi},
for any $h\otimes g\otimes g'\in C(I)\otimes C(\ztwo)\otimes C(\ztwo)$, we determine the formulae:
\begin{align}\label{pidef}
(\pi^{02}_1)^{-1}\left([h\otimes g\otimes g']^{02}_1\right)&=[\sigma^{-1}(\hat\omega_2(h\otimes 
g))\otimes g']^0_{21},\nonumber\\
(\pi^{12}_0)^{-1}\left([h\otimes g\otimes g']^{12}_0\right)&=[\sigma^{-1}(\hat\omega_2(h\otimes 
g))\otimes g']^1_{20}.
\end{align}
Furthermore, taking $\omega$ to be a linear splitting of 
$\sigma\colon\tplz\rightarrow C(S^1)$, using the notation 
\begin{equation}\label{notswe}
\sigma_1(b)= 1_{C(\ztwo)}\otimes b_{10}+u\otimes b_{11},\quad
\sigma_2(b)=b_{20}\otimes 1_{C(\ztwo)}+b_{21}\otimes u,\quad 
b\in\tplz,
\end{equation}
and employing \eqref{jedi}, \eqref{pidef}, \eqref{omegaprop}, for any $b\otimes g\in\tplz\otimes C(\ztwo)$,
 we compute:
\begin{align}
(\pi^{01}_2\circ(\pi^{02}_1)^{-1}&\circ\pi^{20}_1)([b\otimes g]^2_{01})\nonumber\\
&=[\bigl((\sigma_1\otimes\id)\circ((\omega\circ\hat\omega_2)\otimes\id)\circ\Phi_{02}\circ(\sigma_1\otimes\id)\bigr)
(b\otimes g)]^{01}_2\nonumber\\
&=[\bigl(((\delta^*_1\circ\hat\omega_2)\otimes\id)\circ\Phi_{02}\bigr)
(1_{C(\ztwo)}\otimes b_{10} \otimes g+u\otimes b_{11}\otimes g)]^{01}_2\nonumber\\
&=[\bigl((\delta^*_1\circ\hat\omega_2)\otimes\id\bigr)
(b_{10}\otimes g \otimes 1_{C(\ztwo)} +b_{11}\otimes g \otimes u)]^{01}_2\nonumber\\
&=[\imath_\ztwo^*(b_{10})\otimes \imath_*^\ztwo(g) \otimes 1_{C(\ztwo)} 
+\imath_\ztwo^*(b_{11})\otimes \imath_*^\ztwo(g) \otimes u]^{01}_2\,,
\nonumber\\ %\phantom{a}\nonumber\\
(\pi^{10}_2\circ(\pi^{12}_0)^{-1}&\circ\pi^{21}_0)([b\otimes g]^2_{01})\nonumber\\
&=[\bigl(\Phi_{01}\circ(\sigma_1\otimes\id)\circ((\omega\circ\hat\omega_2)\otimes\id)\circ\Phi_{12}\circ(\sigma_2\otimes\id)\bigr)(b\otimes g)]^{01}_2\nonumber\\
&=[\bigl(\Phi_{01}
\circ((\delta^*_1\circ\hat\omega_2)\otimes\id)\circ\Phi_{12}\bigr)\otimes 1_{C(\ztwo)}\otimes g +b_{21}\otimes u\otimes g)]^{01}_2\nonumber\\
&=[\bigl(\Phi_{01}
\circ((\delta^*_1\circ\hat\omega_2)\otimes\id)\bigr)
(b_{20}\otimes g\otimes 1_{C(\ztwo)} +b_{21}\otimes g \otimes u)]^{01}_2\nonumber\\
&=[\Phi_{01}
(\imath_\ztwo^*(b_{20})\otimes \imath_*^\ztwo(g)\otimes 1_{C(\ztwo)} +\imath_\ztwo^*(b_{21})
\otimes \imath_*^\ztwo(g) \otimes u)]^{01}_2\nonumber\\
&=[
1_{C(\ztwo)} \otimes \imath_*^\ztwo(g)\otimes \imath_\ztwo^*(b_{20})
+u\otimes \imath_*^\ztwo(g) \otimes \imath_\ztwo^*(b_{21})]^{01}_2\,.
\end{align}
Hence \eqref{lastmatcond} is satisfied provided that,
\begin{equation}
\imath_\ztwo^*(b_{10}) \otimes 1_{C(\ztwo)} 
+\imath_\ztwo^*(b_{11}) \otimes u
=1_{C(\ztwo)} \otimes \imath_\ztwo^*(b_{20})
+u \otimes \imath_\ztwo^*(b_{21}),\quad \forall\; b\in\tplz.
\end{equation}
Remembering \eqref{notswe} and applying the flip to the above equation, one sees that it is equivalent to
$(\id\otimes\imath_\ztwo^*)\circ\sigma_1=(\imath_\ztwo^*\otimes\id)\circ\sigma_2$. Due to the surjectivity of $\sigma$, the latter is tantamount
to
 $(\id\otimes\imath_\ztwo^*)\circ\delta^*_1=(\imath_\ztwo^*\otimes\id)\circ\delta^*_2$, which can be immediately verified. 

Thus, we have proven
that, by Proposition~\ref{MattProp}, all maps $\qcube\to (\mathcal{T}\otimes C(\ztwo))_i$ are surjective. Furthermore, they are, by construction, 
$\ztwo$-equivariant for the diagonal action on 
$\mathcal{T}\otimes C(\ztwo)$. The $\ztwo$-equivariance is equivalent
to the $C(\ztwo)$-colinearity for the induced coactions. 
Using the gauge conjugation
by \eqref{gbdef}, we see that $\mathcal{T}\otimes C(\ztwo)$
with the induced diagonal $C(\ztwo)$-coaction
is a trivial principal comodule algebra.
Combining all this with the fact that the kernels of the maps
$\qcube\to (\mathcal{T}\otimes C(\ztwo))_i$
intersect to zero, we take advantage of \cite[Theorem~3.3]{HKMZ} 
to conclude:
\begin{proposition}
$\qcube$ is a principal
$C(\ztwo)$-comodule algebra.
\end{proposition}

\subsection{The tautological line bundle}

The tautological line bundle over $\B RP^2$ can be defined
as the line bundle associated with the $\ztwo$-principal bundle
$S^2\to\B RP^2$ via the antipodal action of $\ztwo$ on $\B C$. 
This antipodal action translates to the coaction given by the
formula $1\mapsto u\otimes 1$, where  $u\in C(\ztwo)$
is defined by $u(\pm 1)=\pm 1$. We can now use this coaction
to associate with the principal $C(\ztwo)$-comodule algebra
$\qcube$ a
finitely generated projective left $C(\proj R2)$-module 
$L:=C(S^2_{\B R\mathcal{T}})\square^{C(\ztwo)}\B C$.
This  is the module
of the noncommutative tautological
line bundle over the quantum projective space $\proj R2$.
Our primary goal is to prove that this bundle is not stably trivial,
i.e.\ that $L$ is not stably free.

In order to determine the $K_0$-class of $L$, we need refer to yet
another isomorphic construction of $C(\proj R2)$. Let us  recall that 
$\B R P^2:=S^2/(\ztwo)$ is homeomorphic with a disc whose boundary circle is divided by the antipodal $\ztwo$-action. In the same
spirit, we will show that $C(\proj R2)$ and $L$ are respectively
isomorphic with
$C(D_\tplz)^+$ and $C(D_\tplz)^-$ of~\eqref{cdt+-}. First, we define
 $C(D_\mathcal{T})$ that will play the role of the C*-algebra
of a disk in the above construction:
\[\label{cdt}
\xymatrix{&&C(D_\mathcal{T})\ar[dll]\ar[d]\ar[drr]&&\\
\mathcal{T}_0\ar[dr]\ar@/_2pc/[ddrr]&&\mathcal{T}_1\ar[dr]\ar[dl]&&\mathcal{T}_2\,,\ar[dl]\ar@/^2pc/[ddll]\\
&C(I)&&C(I)&\\
&&C(I)&&}
\]
\begin{eqnarray}\label{cdtf}C(D_\mathcal{T}):=\{(p_0,p_1,p_2)\in\mathcal{T}^3 | \sigma_1(p_0)(-1,x)&=&\sigma_1(p_1)(-1,x),
\nonumber\\
\sigma_2(p_0)(x,-1)&=&\sigma_1(p_2)(-1,x),
\nonumber\\ 
\sigma_2(p_1)(x,-1)&=&\sigma_2(p_2)(x,-1)\}.
\end{eqnarray}

Throughout this section, we will frequently consider a $\ztwo$-action
on an algebra~$C(\#)$:
\[
\alpha^\#\colon\ztwo\longrightarrow \mathrm{Aut}(C(\#)),
\qquad \alpha^\#_{-1}:=\alpha^\#(-1).
\]
In particular, $\alpha^{\ztwo}_{-1}$ is simply
the pullback of the multiplication by $-1$. With the help of
this notation, we define
\[\label{qcube+-}
C(S^2_{\B R\mathcal{T}})^\pm:=\{(p_i\otimes t_i)_i\in C(S^2_{\B R\mathcal{T}})\;|\;\alpha^{\mathcal{T}}_{-1}(p_i)\otimes 
\alpha^{\ztwo}_{-1}(t_i)=\pm p_i\otimes t_i\text{ for } i=0,1,2\}.
\]
Note that $C(S^2_{\B R\mathcal{T}})^+
=C(S^2_{\B R\mathcal{T}})^{\mathrm{co}\,C(\ztwo)}$ and
$L$ is naturally isomorphic with $C(S^2_{\B R\mathcal{T}})^-$
(by ommitting $\otimes 1$).
Thus $+$ and $-$ stand for the $\ztwo$-invariant and 
$\ztwo$-equivariant part, respectively. 

Next, we shall argue that
$C(S^2_{\B R\mathcal{T}})$ can be identified with the pullback
C*-algebra of the following diagram:
\[\label{diagram}
\xymatrix{&C(S^2_{\B R\mathcal{T}})=C(S^2_{\B R\mathcal{T}})^{+}\oplus C(S^2_{\B R\mathcal{T}})^{-}\ar[dr]^{\pi_2}\ar[dl]_{\pi_1}&\\
C(D_\mathcal{T})\ar[dr]^{\sigma_1^\circ}&&C(D_\mathcal{T})\,.\ar[dl]_{\sigma_2^\circ}\\
&C(S^1)&}
\]
In this diagram, the top maps are defined as 
\[\label{pi+-}
\pi_n\colon C(S^2_{\B R\mathcal{T}})\ni(p_i\otimes t_i)\longmapsto (\alpha_{(-1)^n}^\tplz (p_i)t_i((-1)^{n+1}))\in C(D_\mathcal{T})\text{ for } n=1,2.
\]
To specify the maps  $\sigma_n^\circ$,
first we  identify six continuous functions on intervals that agree on appropriate endpoints with a continuous function on a circle.
One sees that the antipodal action on $S^1$ pullbacks to
\[
\alpha_{-1}^{S^1}\colon C(S^1)\ni(f_1,\dots,f_6)\longmapsto(f_4,f_5,f_6,f_1,f_2,f_3)\in C(S^1).
\]
This map reflects the difference between the way in which the left 
$D_\mathcal{T}$ and the right $D_\mathcal{T}$ are embedded in 
$S^2_{\B R\mathcal{T}}$. Now we can define 
$\sigma^\circ_2:=\alpha_{-1}^{S^1}\circ\sigma^\circ_1$ and 
\begin{align}\label{scirc}
\sigma^\circ_1(p_0,p_1,p_2):=
\bigl(((&\ev_1\otimes\id)\circ\sigma_1)(p_0)\,,\,((\alpha^I_{-1}\otimes\ev_1)\circ\sigma_2)(p_0),\\ 
((&\id\otimes\ev_1)\circ\sigma_2)(p_1)\,,\,((\ev_1\otimes\alpha^I_{-1})\circ\sigma_1)(p_1),\nonumber\\ 
((&\ev_1\otimes\id)\circ\sigma_1)(p_2)\,,\,((\alpha^I_{-1}\otimes\ev_1)\circ\sigma_2)(p_2)\bigr).\nonumber
\end{align}
These definitions ensure the commutativity of the diagram \eqref{diagram}, i.e.\
$\sigma^\circ_1\circ\pi_1=\sigma^\circ_2\circ\pi_2$. Hence, we have
a $*$-homomorphism 
\[
C(S^2_{\B R\mathcal{T}})\ni x\longmapsto (\pi_1(x),\pi_2(x))\in
\text{ the pullback C*-algebra 
of $\sigma^\circ_1$ and $\sigma^\circ_2$}\,.
\]
It is straightforward to verify that the above map is bijective, so
 $C(S^2_{\B R\mathcal{T}})$ is isomorphic with the pullback C*-algebra of the diagram~\eqref{diagram}.

It is easily checked that the compositions $\sigma_n^\circ\circ\pi_n$ are $\ztwo$-equivariant with respect to the antipodal actions on $C(S^2_{\B R
\mathcal{T}})$ and $C(S^1)$. 
 Indeed, on the left part of the following picture 
(see~\eqref{nota} for notation) the
antipodal $\ztwo$-action on $C(S^2_{\B R\mathcal{T}})$ 
restricted to $\sigma^\circ_1(C(D_\mathcal{T}))$ coincides with the above defined 
antipodal $\ztwo$-action on $C(S^1)$ (see the right figure below).
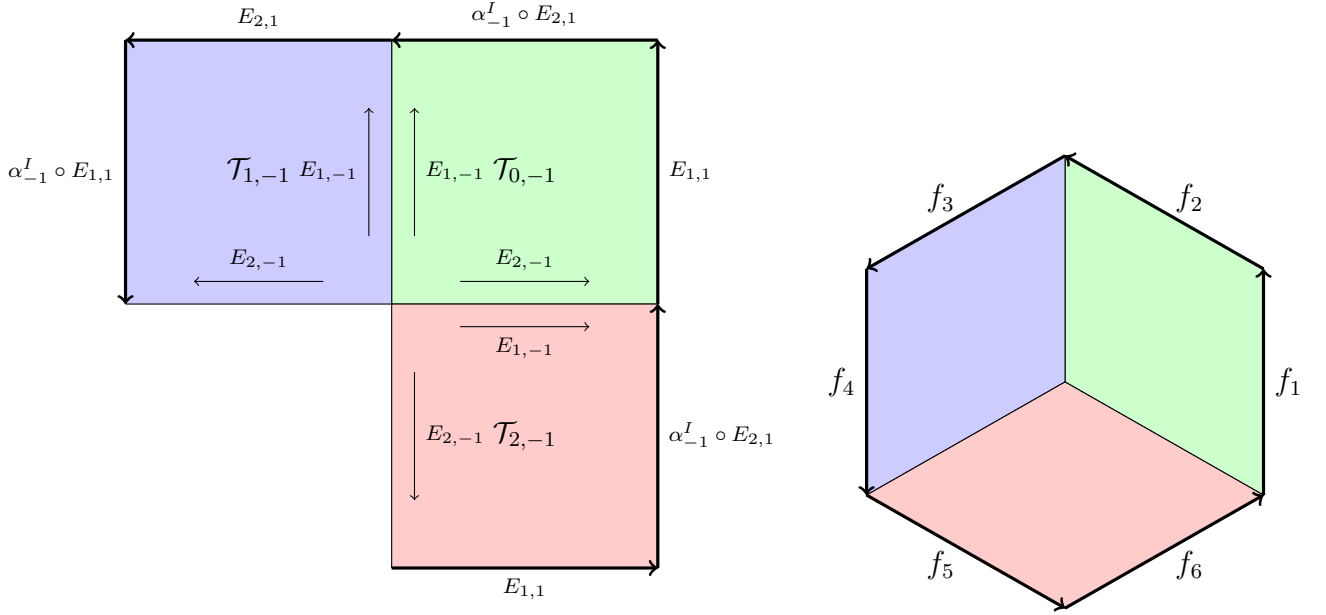
\[
\newcommand{\dd}{3.5cm}
\newcommand{\ddh}{1.75cm}
\newcommand{\ead}{3mm}
\newcommand{\ebd}{3.2cm}
\newcommand{\eada}{0.9cm}
\newcommand{\ebda}{2.6cm}
\begin{tikzpicture}
\draw[fill=green!20!white,draw=black] (0,0)--(\dd,0)--(\dd,\dd)--(0,\dd)--(0,0);
\draw[fill=blue!20!white,draw=black,xshift=-\dd] (0,0)--(\dd,0)--(\dd,\dd)--(0,\dd)--(0,0);
\draw[fill=red!20!white,draw=black,yshift=-\dd] (0,0)--(\dd,0)--(\dd,\dd)--(0,\dd)--(0,0);
\draw (1.75cm,1.75cm) node{\small $\tplz_{0,-1}$};
\draw[xshift=-\dd] (\ddh,\ddh) node{\small $\tplz_{1,-1}$};
\draw[yshift=-\dd] (\ddh,\ddh) node{\small $\tplz_{2,-1}$};

\draw[->,very thick] (\dd,0)--(\dd,\dd);
\draw[->,very thick] (\dd,\dd)--(0,\dd);
\draw[->,very thick] (0,\dd)--(-\dd,\dd);
\draw[->,very thick] (-\dd,\dd)--(-\dd,0);
\draw[->,very thick] (0,-\dd)--(\dd,-\dd);
\draw[->,very thick] (\dd,-\dd)--(\dd,0);

\draw[->] (\ead,\eada)--(\ead,\ebda);

\draw[->] (\eada,\ead)--(\ebda,\ead);

\draw[<-,xshift=-\dd] (\eada,\ead)--(\ebda,\ead);
\draw[->,xshift=-\dd] (\ebd,\eada)--(\ebd,\ebda);

\draw[<-,yshift=-\dd] (\ead,\eada)--(\ead,\ebda);
\draw[->,yshift=-\dd] (\eada,\ebd)--(\ebda,\ebd);

\draw (\ead,\ddh) node[anchor=west]{\tiny $E_{1,-1}$};
\draw (\dd,\dd/2) node[anchor=west]{\tiny $E_{1,1}$};
\draw (\ddh,\ead) node[anchor=south]{\tiny $E_{2,-1}$};
\draw (\dd/2,\dd) node[anchor=south]{\tiny $\alpha_{-1}^I \circ E_{2,1}$};

\draw[xshift=-\dd] (0,\dd/2) node[anchor=east]{\tiny $\alpha_{-1}^I\circ E_{1,1}$};
\draw[xshift=-\dd] (\ebd,\ddh) node[anchor=east]{\tiny $E_{1,-1}$};
\draw[xshift=-\dd] (\ddh,\ead) node[anchor=south]{\tiny $E_{2,-1}$};
\draw[xshift=-\dd] (\dd/2,\dd) node[anchor=south]{\tiny $E_{2,1}$};

\draw[yshift=-\dd] (\ead,\ddh) node[anchor=west]{\tiny $E_{2,-1}$};
\draw[yshift=-\dd] (\dd,\dd/2) node[anchor=west]{\tiny $\alpha_{-1}^I\circ E_{2,1}$};
\draw[yshift=-\dd] (\dd/2,0) node[anchor=north]{\tiny $E_{1,1}$};
\draw[yshift=-\dd] (\ddh,\ebd) node[anchor=north]{\tiny $E_{1,-1}$};
\end{tikzpicture}\quad
\newcommand{\sss}{3*0.87cm}
\begin{tikzpicture}
\draw (\sss,0) node[anchor=west]{$f_1$};
\draw (\sss/2,2.1) node[anchor=south west]{$f_2$};
\draw (-\sss/2,2.1) node[anchor=south east]{$f_3$};
\draw (-\sss,0) node[anchor=east]{$f_4$};
\draw (-\sss/2,-2.1) node[anchor=north east]{$f_5$};
\draw (\sss/2,-2.1) node[anchor=north west]{$f_6$};

\draw[fill=green!20!white,draw=black] (0,0)--(\sss,-3/2)--(\sss,3/2)--(\sss,3/2)--(0,3)--(0,0);
\draw[fill=blue!20!white,draw=black]
(0,0)--(-\sss,-3/2)--(-\sss,3/2)--(-\sss,3/2)--(0,3)--(0,0);
\draw[fill=red!20!white,draw=black]
(0,0)--(-\sss,-3/2)--(0,-3)--(\sss,-3/2)--(0,0);

\draw[->,very thick] (\sss,-3/2)--(\sss,3/2);
\draw[->,very thick] (\sss,3/2)--(0,3);
\draw[->,very thick] (0,3)--(-\sss,3/2);
\draw[->,very thick] (-\sss,3/2)--(-\sss,-3/2);
\draw[->,very thick] (-\sss,-3/2)--(0,-3);
\draw[->,very thick] (0,-3)--(\sss,-3/2);

\end{tikzpicture}
\]

Since $L\cong C(S^2_{\B R\mathcal{T}})^-$, our next step is
to transform $C(S^2_{\B R\mathcal{T}})^-$ to a more managable form. 
To this end, using \eqref{scirc} and the line above it, we define
\[\label{cdt+-}
C(D_\mathcal{T})^\pm:=\{(p_0,p_1,p_2)\in C(D_\mathcal{T})\;|\;
\sigma^\circ_2(p_1,p_2,p_3)=\pm \sigma^\circ_1(p_1,p_2,p_3)\}.
\]
Next, we
note  that it follows from the $\ztwo$-equivariance of 
$\sigma_n^\circ\circ\pi_n$ that $\pi_n^\pm(C(S^2_{\B R\mathcal{T}})^\pm)\subseteq 
C(D_\mathcal{T})^\pm$, so 
the restrictions of the 
$*$-ho\-mo\-mor\-phisms \eqref{pi+-} define
\[\label{pi+-n}
\pi_n^\pm:C(S^2_{\B R\mathcal{T}})^\pm\longrightarrow C(D_\mathcal{T})^\pm, n\in\{1,2\}. 
\]
\begin{lemma}\label{isolema} 
Let $n\in\{1,2\}$.
The restrictions  
$\pi_n^+$ are isomorphisms of C*-algebras, and $\pi_n^-$
are isomorphisms of modules over $C(S^2_{\B R\mathcal{T}})^+$.
\end{lemma}
\begin{proof}
We consider only the case $n=1$ as the case $n=2$ is analogous. 
Let us define
\begin{gather}
(\pi_1^\pm)^{-1}:C(D_\mathcal{T})^\pm\ni(p_0,p_1,p_2)\longmapsto(\textbf{p}_0,\textbf{p}_1,\textbf{p}_2)\in C(S^2_{\B R\mathcal{T}})^\pm,
\nonumber\\
\textbf{p}_i:=\alpha^{\mathcal{T}}_{-1}(p_i)\otimes \textbf{1}_1\pm p_i\otimes\textbf{1}_{-1},
\end{gather}
where $\textbf{1}_x$ is the function taking $1$ at $x$ and $0$ everywhere else. To show that the ranges of $(\pi_1^\pm)^{-1}$ are
indeed $C(S^2_{\B R\mathcal{T}})^\pm$, respectively,
first we need to check that 
$(\pi_1^\pm)^{-1}(C(D_\mathcal{T})^\pm)\subseteq \qcube$.
To verify this inclusion, we have to check that the defining equalities
\eqref{qcube} hold. 
We will do this only for the first equality 
\[\label{cel}
(\sigma_1\otimes\id)(\alpha^{\mathcal{T}}_{-1}(p_0)\otimes \textbf{1}_1
\pm p_0\otimes\textbf{1}_{-1})=\bigl(\Phi_{01}\circ(\sigma_1\otimes\id)\bigl)(\alpha^{\mathcal{T}}_{-1}(p_1)\otimes \textbf{1}_1
\pm p_1\otimes\textbf{1}_{-1})
\]
as the remaing ones are similar. If $(p_0,p_1,p_2)\in C(D_\mathcal{T})^\pm$, it follows
 from \eqref{cdtf} and \eqref{cdt+-} that
\begin{align}\label{podsu}
\bigl((\ev_{-1}\otimes\id)\circ\sigma_1\bigr)(p_0)&=\bigl((\ev_{-1}\otimes\id)\circ\sigma_1\bigr)(p_1),
\\
\bigl((\ev_1\otimes\id)\circ\sigma_1\bigr)(p_0)&=\pm\bigl((\ev_1\otimes\alpha_{-1}^I)\circ\sigma_1\bigr)(p_1).
\nonumber\end{align} 
Next, let us  introduce the following Heynemann-Sweedler-type notation
with the summation sign suppressed:
\[
\sigma_1(p)=\sigma_1(p)^{(1)}\otimes\sigma_1(p)^{(0)},\quad
\sigma_2(p)=\sigma_2(p)^{(0)}\otimes\sigma_1(p)^{(1)}.
\]
Now, remembering the $\ztwo$-equivariance of $\sigma_1$ and 
$\sigma_2$, we transform \eqref{cel}  into the following 
equivalent form:
\begin{align}
&\alpha_{-1}^\ztwo(\sigma_1(p_0)^{(1)})\otimes\alpha_{-1}^I(\sigma_1(p_0)^{(0)})\otimes\textbf{1}_1\;\pm\;
\sigma_1(p_0)^{(1)}\otimes\sigma_1(p_0)^{(0)}\otimes\textbf{1}_{-1}\\
&=\textbf{1}_1\otimes\alpha_{-1}^I(\sigma_1(p_1)^{(0)})\otimes\alpha_{-1}^\ztwo(\sigma_1(p_1)^{(1)})\;\pm\;
\textbf{1}_{-1}\otimes\sigma_1(p_1)^{(0)}\otimes\sigma_1(p_1)^{(1)}.\nonumber
\end{align}
One can directly check that this formula holds  by evaluating 
the outside legs on the elements of $\ztwo\times\ztwo$  and 
 using~\eqref{podsu}. Finally, the fact that 
$C(D_\tplz)^\pm$ are mapped, respectively, to $\qcube^\pm$ 
follows immediately
from the definition of $\qcube^\pm$ (see~\eqref{qcube}).

We are ready now to verify 
 that both compositions $\pi^\pm_1\circ(\pi_1^{\pm})^{-1}$ and $(\pi_1^{\pm})^{-1}\circ\pi_1^\pm$ are equal to identity. 
First, for each component of $C(D_\tplz)^\pm$
we check that
\[
\pi_1^\pm\bigl(\alpha^{\mathcal{T}}_{-1}(p)\otimes \textbf{1}_1\pm p\otimes\textbf{1}_{-1}\bigr)= \alpha_{-1}^\tplz(\alpha_{-1}^\tplz (p))\;\textbf{1}_1(1)=p\,.
\]
Hence, $\pi^\pm_1\circ(\pi_1^{\pm})^{-1}=\id$. To see the other identity,
we compute:
\begin{align}
(\pi_1^\pm)^{-1}\bigl(\alpha_{-1}^\tplz (p)t(1)\bigr)
&=\alpha^{\mathcal{T}}_{-1}(\alpha_{-1}^\tplz (p))t(1)\otimes \textbf{1}_1\pm \alpha_{-1}^\tplz (p)t(1)\otimes\textbf{1}_{-1}
\nonumber\\
&=p\,t(1)\otimes \textbf{1}_1\pm p\, (\pm\alpha_{-1}^\ztwo (t))(1)\otimes\textbf{1}_{-1}
\nonumber\\
&=p\otimes \bigl( t(1)\textbf{1}_{1}+ t(-1)\textbf{1}_{-1}\bigr)
\nonumber\\
&=p\otimes t\,.
\end{align}
Here to pass from the first to the second line we used the fact 
that
\[
\alpha^{\mathcal{T}}_{-1}(p)\otimes 
\alpha^{\ztwo}_{-1}(t)=\pm p\otimes t \quad\Longrightarrow\quad
\alpha^{\mathcal{T}}_{-1}(p)\otimes t=\pm p\otimes 
\alpha^{\ztwo}_{-1}(t).
\]

To end with, observe that $\pi_1^-$ is an isomorphisms of modules in the sense that $\pi_1^-(av)=\pi_1^+(a)\pi_1^-(v)$.
\end{proof}
To prove that $L\cong C(D_\mathcal{T})^-$ is not stably free, we will
proceed along the lines of~\cite{bh}, where it was crucial to use
the fact that $K_1(\tplz)=0$. Here it is $D_\tplz$ that plays the
role of $\tplz$.
\begin{lemma}$K_0(C(D_{\mathcal{T}}))\cong\B Z,\qquad K_1(C(D_{\mathcal{T}}))\cong 0.$\label{kdysk}
\end{lemma}
\begin{proof}
In Section~\ref{2.3}, we have proven that all
maps $\qcube\to (\mathcal{T}\otimes C(\ztwo))_i$ are surjective.
Combining this with \eqref{pi+-}, one can easily conclude
that all restrictions to $C(D_\tplz)$ of the canonical surjections
are also surjective. Therefore, we can
use \cite[Lemma~0.2]{r-j12} to
 convert the defining  triple-pullback diagram \eqref{cdt}
to the iterated pullback diagram:
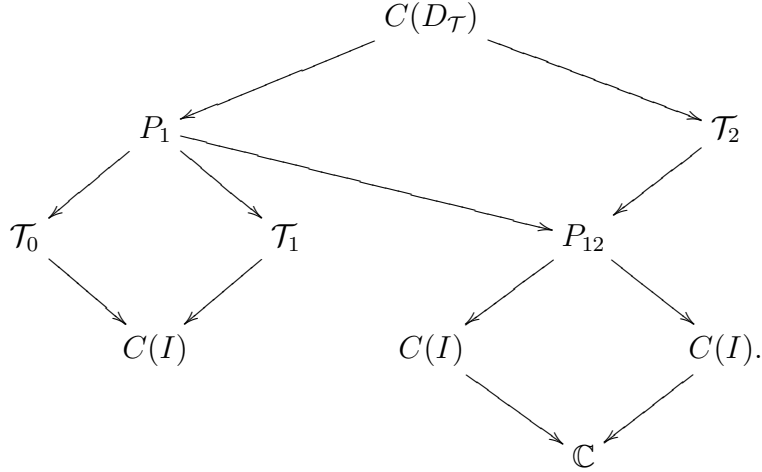
\[\label{iteratedp}
\xymatrix{
&&&C(D_{\mathcal{T}})\ar[dll]\ar[drr]&&\\
&P_1\ar[dr]\ar[drrr]\ar[dl]&&&&\mathcal T_2\ar[dl]\\
\mathcal T_0\ar[dr]&&\mathcal T_1\ar[dl]&& P_{12}\ar[dl]\ar[dr]&\\
&C(I)&&C(I)\ar[dr]&&C(I).\ar[dl]\\
&&&&\B C&}
\]
Here $I$ is identified with an arc of $S^1$ as previously done
(see~\eqref{pbelow}). Next, applying the Mayer--Vietoris six-term
exact sequence to the bottom pullback sub-diagrams of the above diagram,
we obtain
\[
\xymatrix{K_0(P_1)\ar[r]&\B Z\oplus \B Z\ar[r]&\B Z\ar[d]\\
0\ar[u]&0\ar[l]&K_1(P_1),\ar[l]}\qquad
\xymatrix{K_0(P_{12})\ar[r]&\B Z\oplus \B Z\ar[r]&\B Z\ar[d]\\
0\ar[u]&0\ar[l]&K_1(P_{12}).\ar[l]}
\]
Since $K_0(\tplz)\cong\B Z\cong K_0(C(I))$ are generated by the classes 
of  respective $1$'s in the algebras, both arrows 
$\B Z\oplus\B Z\to\B Z$ are given by the formula 
$(a,b)\mapsto a-b$. Hence, we obtain
\begin{align*}
K_0(P_1)=\B Z&,\quad K_0(P_{12})=\B Z,\\
 K_1(P_1)=0&,\quad K_1(P_{12})=0.
\end{align*}
This in turn yields the following form of the
 Mayer--Vietoris six-term
exact sequence of the top pullback sub-diagram of \eqref{iteratedp}: 
\[
\xymatrix{K_0(C(D_{\mathcal{T}}))\ar[r]&\B Z\oplus \B Z\ar[r]&\B Z\ar[d]\\
0\ar[u]&0\ar[l]&K_1(C(D_{\mathcal{T}})).\ar[l]}
\]
Finally, as the arrow $\B Z\oplus\B Z\to\B Z$ is again (and
for much the same reasons) given
by the formula $(a,b)\mapsto a-b$, we conclude the claim of the lemma.
\end{proof}

\begin{theorem}\label{nst}
 Let $L:=C(S^2_{\B R\mathcal{T}})\square^{C(\ztwo)}\B C$ be the associated left $C(\B R P_\mathcal{T}^2)$-module for the coaction $\varrho\colon\B C\to C(\ztwo)\otimes \B C$, $\varrho(1)=u\otimes 1$, $u(\pm 1)=\pm 1$. Then $L$ is \emph{not stably free}. In other words, the tautological line bundle over $\B R P_\mathcal{T}^2$ is not stably trivial. 
\end{theorem}
\begin{proof}
Suppose that $L$ is stably free. Since $\qcube$ is principal,
it follows from the stable triviality criterion \cite{h-pm} that there exists an invertible matrix $T\in M_n(C(S^2_{\B R\mathcal{T}}))$ whose first row has entries in $L\cong C(S^2_{\B R\mathcal{T}})^{-}$ and all other rows have entries in $C(\B R P^2_\mathcal{T})=C(S^2_{\B R\mathcal{T}})^{+}$. Next, let $T_1:=(\pi_1(T_{ij}))$ (see \eqref{pi+-} for $\pi_1$) be the corresponding invertible matrix over $C(D_\mathcal{T})$. Then, by Lemma~\ref{isolema}, the first row of $T_1$ has entries in 
$C(D_\mathcal{T})^-$ and all other rows have entries in~$C(D_\mathcal{T})^+$. Furthermore, applying 
$\sigma^\circ_1$
of \eqref{scirc} componentwise to $T_1$, we obtain an invertible matrix $T_2$ over~$C(S^1)$. 
It follows directly from the definition of $C(D_\mathcal{T})^\pm$ that the determinant of this matrix is a 
$\ztwo$-equivariant function, i.e.\ $\det(T_2)(-t)=-\det(T_2)(t)$. A standard topological argument shows that 
the winding number of such a function (normalized to a function from $S^1$ to $S^1$) is odd. Hence, the 
$K_1$ class of $T_2$ is odd. Furthermore, this class equals to $\sigma^\circ_{1*}
([T_1]_{K_1(C(D_\mathcal{T}))})$, which  contradicts the fact that $K_1(C(D_\mathcal{T}))=0$ (see 
Lemma~\ref{kdysk}).
\end{proof}

Consider now the obvious Hopf algebra surjection
$\pi\colon \mathcal{O}(U(1))\rightarrow C(\ztwo)$. This yields
the pro\-lon\-ga\-tion $\qcube\square^{C(\ztwo)}\mathcal{O}(U(1))$
(cf.~\cite{bh}). Since $\qcube$ is a principal $C(\ztwo)$-comodule
algebra,
it follows from
   \cite[Lemma~2.3]{bz12} that
\mbox{$\qcube\square^{C(\ztwo)}{\mathcal{O}}(U(1))$}  is  a  principal
${\mathcal{O}}(U(1))$-comodule algebra.
Furthermore, as 
$L:=C(S^2_{\B R\mathcal{T}})\square^{C(\ztwo)}\B C$ is not free
due to Theorem~\ref{nst}, we conclude that the $C(\ztwo)$-comodule algebra $\qcube$ is not cleft. Likewise, since 
\[
\qcube\square^{C(\ztwo)}\C\cong
\qcube\square^{C(\ztwo)}\mathcal{O}(U(1))\square^{\mathcal{O}(U(1))}\C,
\]  
we can view $L$ as a module associated to the $\mathcal{O}(U(1))$-comodule algebra $\qcube\square^{C(\ztwo)}\mathcal{O}(U(1))$.
Hence, the latter is also not cleft. Next, since the 
$\qcube$ and $\widetilde{\qcube}$ are isomorphic as $C(\ztwo)$-comodule algebras, and
the latter is a piecewise-trivial principal $C(\ztwo)$-comodule algebra
by construction, so is~$\qcube$.  Combining
this with Lemma~\ref{almostinv}
and the obvious fact that ${\mathcal{O}}(U(1))$ is a principal
$C(\ztwo)$-comodule algebra, we can apply Theorem~\ref{mainres}
to conclude that $\qcube\square^{C(\ztwo)}{\mathcal{O}}(U(1))$
admits a $\ker\pi$-reduction. Thus, we obtain a non-trivial illustration
of Theorem~\ref{mainres}: a non-cleft piecewise-trivial and reducible
principal comodule algebra.

\section{The irreducibility of a quantum-plane frame bundle}
\noindent
The aim of this section is to show that the frame bundle of the quantum
plane $\mathbb{C}_q$ 
is not reducible to an $SL_q(2)$-subbundle unless $q$ is a cubic root
of 1 \cite{hm98}. To this end, we will need: 

\begin{proposition}\label{trivred}
For a smash product $P=B\rtimes H$, the elements
 $f\in\alg^H_H(\lco,Z_P(B))$  are in bijective correspondence
with unital linear maps $\vartheta:\lco\rightarrow B$ satisfying,
 for all $k,l\in\lco$, $h\in H$, $b\in B$,
\begin{gather}\label{redprop}
\vartheta(kl)=\vartheta(l)\vartheta(k),\quad
b\vartheta(k)=\vartheta(k_{(1)})(k_{(2)}\triangleright b),\quad 
\vartheta(Sh_{(1)}kh_{(2)})=Sh\triangleright \vartheta(k).
\end{gather}
The correspondence is given explicitly by
\begin{equation}\label{redcor}
		  f \longmapsto \vartheta_f = (\mathrm{id}_B \otimes 
			\counit) \circ f,\quad
		\vartheta \longmapsto f_\vartheta =
		(\vartheta \otimes \mathrm{id}_H) \circ \comul.  	
\end{equation}
\end{proposition}
\begin{proof}
The correspondence \eqref{redcor} can be proven 
using the right $H$-colinearity of $f$. Next,
put $D:=\lco$. 
Then $bf(k)=f(k)b$ for all $k\in D$ and $b\in B$. Explicitly,
\begin{equation}
bf(k)=b\vartheta(k\sw{1})\otimes k\sw{2}\quad\quad\text{and}\quad\quad
 f(k)b=\vartheta(k\sw{1})(k\sw{2}\triangleright b)\otimes k\sw{3}.
\end{equation}
Hence, the second equality in \eqref{redprop} follows. In order to prove the first one, we use 
the fact that $f$ is an algebra homomorphism. For any $k,l\in D$, 
we have $f(kl)=\vartheta(k\sw{1}l\sw{1})\otimes k\sw{2}l\sw{2}$. 
Furthermore,
\begin{equation}
f(kl)=f(k)f(l)=(\vartheta(k\sw{1})\otimes k\sw{2})(\vartheta(l\sw{1})\otimes l\sw{2})
=\vartheta(k\sw{1})(k\sw{2}\triangleright\vartheta(l\sw{1}))\otimes k\sw{3}l\sw{2}.
\end{equation}
Therefore,  the already proven second property of \eqref{redprop} 
and the fact that $\vartheta(l)\in B$ yield
\begin{equation}
\vartheta(kl)=\vartheta(k\sw{1})(k\sw{2}\triangleright\vartheta(l))=
\vartheta(l)\vartheta(k).
\end{equation}
Finally, the last property of $\vartheta$ follows from the 
invariance of $f$ with respect to the Miyashita--Ulbrich $H$-action.
We end this proof by noting that using the above arguments backwards 
shows that, if the map $\vartheta:D\rightarrow B$ satisfies 
\eqref{redprop}, then
the map $k\mapsto\vartheta(k\sw{1})\otimes k\sw{2}$
belongs to $\alg^H_H(\lco,Z_{B\rtimes H}(B))$.\hfill
\end{proof}

We are now ready to 
demonstrate that $B\rtimes H$, where $B=\mathcal{O}({\C}^2_q)$ and 
$H=\mathcal{O}(GL_{q}(2))$,
is not reducible to a principal $\mathcal{O}(SL_{q}(2))$-comodule algebra unless $q^3=1$.
Recall that
$\mathcal{O}({\C}^2_q)$ is defined as the unital associative algebra over ${\C}$ 
generated by $x,y$ with relations
\begin{equation}\label{cp2} 
xy=qyx,\quad q\in{\C}\setminus\{0\},
\end{equation}
and $\mathcal{O}(GL_{q}(2))$ is defined as the unital associative algebra over $\C$
generated by $a,b,c,d,D^{-1}$ with relations
\begin{gather}
ab=qba,\quad ac=qca,\quad bd=qdb,
\quad cd=qdc,\quad bc=cb,\quad ad=da+(q-q^{-1})bc,\label{rel1}\\
\label{rel3} 
(ad-qbc)D^{-1}=D^{-1}(ad-qbc)=1,
\end{gather}
where $q\in{\C}\setminus\{0\}$.
The Hopf algebra structure of $\mathcal{O}(GL_{q}(2))$ is defined in terms of the matrix
\[
\left(\begin{array}{cc}a & b \\ c & d \end{array}\right)
\]
of generators in the usual way.

There exists a well-defined left action of $\mathcal{O}(GL_{q}(2))$ on $\mathcal{O}({\C}^2_q)$ 
given by the formulas
\begin{gather}\label{wx} 
a\triangleright x=q^{-2}x, ~b\triangleright x=0, 
~c\triangleright x=(q^{-2}-1)y, ~d\triangleright x=q^{-1}x, ~D^{-1}\triangleright x=q^3x,
\\
\label{wy} 
a\triangleright y=q^{-1}y, ~b\triangleright y=0, ~c\triangleright y=0, ~d\triangleright y=q^{-2}y, ~D^{-1}
\triangleright y=q^3y.
\end{gather}
Denote by $\pi:\mathcal{O}(GL_{q}(2))\rightarrow \mathcal{O}(SL_{q}(2))$ 
the natural surjection sending $D$ to $1$.
Suppose that there exists a $\ker\pi$-reduction of $B\rtimes H$. 
It follows from Lemma~\ref{trivred}  that
there exists a unital and anti-algebra map
 $\vartheta:{}^{\co A(SL_{q}(2))}H\rightarrow B$.
In particular, as $D,D^{-1}\in{}^{\co A(SL_{q}(2))}H$ and
\begin{equation}
1=\vartheta(1)=\vartheta(DD^{-1})=\vartheta(D^{-1})\vartheta(D),\quad 
1=\vartheta(1)=\vartheta(D^{-1}D)=\vartheta(D)\vartheta(D^{-1}),
\end{equation}
we obtain
 that $\vartheta(D^{-1})$ is an invertible element of $B=\mathcal{O}({\C}^2_q)$.
Since the only invertible elements of $\mathcal{O}({\C}^2_q)$ are multiples of 
identity, we conclude 
that $\vartheta(D^{-1})=\mu 1_B$, with $0\neq \mu\in \C$.
Furthermore, from Lemma~\ref{trivred} and~\eqref{wx}, we obtain
 that
\begin{equation}
\mu x=x\vartheta(D^{-1})=\vartheta(D^{-1})(D^{-1}\triangleright x)
=q^3\mu x,
\end{equation}
so  $q^{3}=1$, as claimed. 

\noindent{\bf Acknowledgements:}
 The authors are grateful to Tomasz Brzezi\'nski for many helpful discussions. They  also 
thank the Department of Mathematics, Swansea University, for hospitality.
This work is part of the EU-project {\sl Geometry and
symmetry of quantum spaces} PIRSES-GA-2008-230836.
 It was  partially supported by the
Polish Government grants N201 1770 33  and 1261/7.PRUE/2009/7.

%\nocite{*}
%
%\bibliographystyle{acm}
%\bibliography{reduct}

\end{document}